\documentclass[12pt,final]{amsart}
\usepackage{amssymb,euscript,nicefrac}
\usepackage{amscd,txfonts}
\usepackage{xr,xr-hyper}
\usepackage[notref,notcite]{showkeys}
\usepackage{latexsym,todonotes}
\usepackage{epsfig,pxfonts,extraipa}
\usepackage{amsfonts}
\usepackage{xypic}
\usepackage{xr,xr-hyper}
\usepackage{bbm,ifpdf}
\ifpdf
  \usepackage{pdfsync,hyperref}
\fi

\usepackage[margin=3cm,footskip=25pt,headheight=21pt]{geometry}
\usepackage{fancyhdr,mathrsfs}
\newtheorem{theorem}{Theorem}[section]
\newtheorem{lemma}[theorem]{Lemma}
\newtheorem{proposition}[theorem]{Proposition}
\newtheorem{corollary}[theorem]{Corollary}

\newtheorem{itheorem}{Theorem}
\newtheorem{definition}[theorem]{Definition}

{
\theoremstyle{remark}

\newtheorem{example}[theorem]{Example}

\newtheorem{remark}[theorem]{Remark}
}

\newcommand{\nc}{\newcommand}
\nc{\func}{\EuScript{T}}

\begin{document}
\numberwithin{equation}{section}
\renewcommand{\theequation}{\arabic{section}.\arabic{equation}}
\newcounter{subeqn}
\renewcommand{\thesubeqn}{\theequation\alph{subeqn}}
\newcommand{\subeqn}{%
  \refstepcounter{subeqn}
  \tag{\thesubeqn}
}\newcommand{\newseq}{%
  \refstepcounter{equation}
  \setcounter{subeqn}{0}
}

\makeatletter
\@addtoreset{subeqn}{equation}

\nc{\cat}{\mathcal{V}}
\nc{\res}{\operatorname{res}}
\newcommand{\Spec}{\operatorname{Spec}}
\newcommand{\Fun}{\operatorname{Fun}}
\newcommand{\stab}{\operatorname{stab}}
\newcommand{\Vol}{\operatorname{Vol}}
\renewcommand{\Re}{\operatorname{Re}}
\renewcommand{\Im}{\operatorname{Im}}
\newcommand{\im}{\operatorname{im}}
\newcommand{\id}{\operatorname{id}}
\newcommand{\Coker}{\operatorname{Coker}}
\newcommand{\Ker}{\operatorname{Ker}}
\newcommand{\codim}{\operatorname{codim}}
\renewcommand{\dim}{\operatorname{dim}}
\newcommand{\Ann}{\operatorname{Ann}}
\newcommand{\Sym}{\operatorname{Sym}}
\newcommand{\rk}{\operatorname{rk}}
\newcommand{\crk}{\operatorname{crk}}

\newcommand{\norm}[1]{\Vert #1 \Vert}
\newcommand{\grad}{\nabla}
\newcommand{\I}{\mathbf{i}}
\newcommand{\Bi}{\mathbf{i}}
\newcommand{\Bj}{\mathbf{j}}

\nc{\slehat}{\mathfrak{\widehat{sl}}_e}
\nc{\sllhat}{\mathfrak{\widehat{sl}}_\ell}
\nc{\glehat}{\mathfrak{\widehat{gl}}_e}
\nc{\slnhat}{\mathfrak{\widehat{sl}}_n}
\nc{\glnhat}{\mathfrak{\widehat{gl}}_n}
\nc{\eE}{\EuScript{E}}
\newcommand{\arxiv}[1]{\href{http://arxiv.org/abs/#1}{\tt arXiv:\nolinkurl{#1}}}

\nc{\eF}{\EuScript{F}}
\nc{\fF}{\mathfrak{F}}
\nc{\fE}{\mathfrak{E}}

\newcommand{\K}{\mathbbm{k}}
\newcommand{\ig}{\int_{\gamma}}
\newcommand{\bin}[2]{
 \left(
   \begin{array}{@{}c@{}}
      #1 \\ #2
  \end{array}
 \right)  }
\newcommand{\cq}{/\!\!/}
\newcommand{\td}[1]{\tilde{#1}}
\newcommand{\into}{\hookrightarrow}
\newcommand{\Z}{\mathbb{Z}}
\newcommand{\Q}{\mathbb{Q}}
\nc{\Qlb}{\mathbb{\bar \Q}_\ell}
\nc{\Fq}{\mathbb{F}_q}
\nc{\Fqb}{\mathbb{\bar F}_q}
\newcommand{\Qt}{\widehat{\mathbb{Q}}}
\newcommand{\om}{\omega}
\newcommand{\bc}{\mathbf{c}}
\newcommand{\N}{\mathbb{N}}
\newcommand{\R}{\mathbb{R}}
\newcommand{\Rt}{\widehat{\mathbb{R}}}

\newcommand{\C}{\mathbb{C}}
\newcommand{\m}{\item}
\newcommand{\wt}{\operatorname{wt}}
\newcommand{\bd}{\partial}
\newcommand{\pf}{\pitchfork}
\newcommand{\ra}{\rightarrow}
\newcommand{\la}{\leftarrow}
\newcommand{\Ra}{\Rightarrow}
\newcommand{\La}{\Leftarrow}
\newcommand{\rla}{\RightLeftarrow}
\newcommand{\g}{\widehat{\mathfrak{g}}}
\newcommand{\h}{\widehat{\mathfrak{h}}}
\newcommand{\fh}{{\mathfrak{h}}}
\newcommand{\sS}{\mathsf{S}} 
\newcommand{\sJ}{\mathsf{J}} 
\newcommand{\sT}{\mathsf{T}}
\newcommand{\M}{\mathfrak{M}}
\newcommand{\fQ}{\mathfrak{Q}}
\nc{\Bv}{\mathbf{v}}
  \nc{\Bw}{\mathbf{w}}
  \nc{\Bb}{\mathbf{b}}
\nc{\tU}{\mathcal{U}}
\nc{\tQ}{\tilde{Q}}
\nc{\Bu}{\mathbf{u}}
 \nc{\Fl}{\operatorname{Fl}}
\nc{\Tr}{\operatorname{Tr}}
\nc{\sheafK}{\EuScript{K}}
\nc{\bmu}{\boldsymbol{\mu}}
\nc{\bpi}{\boldsymbol{\pi}}
\newcommand{\Tmu}{T_\mu}
\newcommand{\Mlm}{M_{\la\mu}}
\newcommand{\Nlm}{N_{\la\mu}}
\newcommand{\Nblm}{\bar N_{\la\mu}}
\newcommand{\Nnm}{N_{\nu\mu}}
\renewcommand{\la}{\lambda}
\newcommand{\cs}{\C^\times}
\newcommand{\subs}{\subseteq}
\newcommand{\htlm}{H^*_{\Tmu}(\Mlm)}
\newcommand{\ihtlm}{I\! H^*_{\Tmu}(\Nblm)}
\newcommand{\hzlm}{H(Z_{\la\mu})}
\newcommand{\Alm}{A_{\la\mu}}
\newcommand{\Atlm}{\tilde A_{\la\mu}}
\newcommand{\Blm}{B_{\la\mu}}
\newcommand{\Btlm}{\tilde B_{\la\mu}}
\newcommand{\Rmu}{R_{\mu}}
\newcommand{\vln}{V_{\la\nu}}
\newcommand{\vlm}{V_{\la\mu}}
\newcommand{\al}{\alpha}
\renewcommand{\b}{\beta}
\renewcommand{\g}{\gamma}
\renewcommand{\d}{\delta}
\newcommand{\surj}{\twoheadrightarrow}
\newcommand{\Hom}{\operatorname{Hom}}
\newcommand{\A}{\mathcal A}
\newcommand{\tmu}{\mathfrak t_{\mu}}
\newcommand{\impl}{\Rightarrow}

\newcommand{\IH}{I\! H}
\newcommand{\IC}{\operatorname{IC}^{\scriptscriptstyle{\bullet}}}
\newcommand{\hookto}{{\hookrightarrow}}
\renewcommand{\(}{\left(}
\renewcommand{\)}{\right)}
\newcommand{\eps}{\epsilon}

\newcommand{\walg}{W}
\newcommand{\dwalg}{\check{W}}

\newcommand{\ab}{{\a\b}}
\newcommand{\bg}{{\b\g}}
\newcommand{\gd}{{\g\d}}
\newcommand{\ag}{{\a\g}}
\renewcommand{\bd}{{\mathbf{d}}}
\newcommand{\ad}{{\a\d}}
\renewcommand{\bg}{{\b\g}}
\newcommand{\abg}{{\a\b\g}}
\newcommand{\abd}{{\a\b\d}}
\newcommand{\agd}{{\a\g\d}}
\newcommand{\bgd}{{\b\g\d}}
\newcommand{\abgd}{{\a\b\g\d}}
\newcommand{\badg}{{\b\a\d\g}}

\newcommand{\cP}{\mathcal{P}}
\newcommand{\Ob}{\operatorname{Ob}}
\newcommand{\cO}{\mathcal{O}}
\renewcommand{\cL}{\mathcal{L}}
\newcommand{\cT}{\mathcal{T}}
\newcommand{\cI}{\mathcal{I}}
\newcommand{\tcO}{\tilde{\cQ}^+}

\newcommand{\CST}{C_{\sS,\sT}}
\newcommand{\av}{\a^\vee}
\newcommand{\bv}{\b^\vee}
\newcommand{\dv}{\d^\vee}
\newcommand{\gv}{\g^\vee}
\newcommand{\becircled}{\mathaccent "7017}
\newcommand{\Aff}{\operatorname{Aff}}
\newcommand{\Ext}{\operatorname{Ext}}
\newcommand{\cS}{\mathcal{S}}
\newcommand{\cF}{\mathcal{F}}
\renewcommand{\cD}{\mathcal{D}}
\newcommand{\bS}{\mathbb{S}}
\newcommand{\bT}{\mathbb{T}}
\newcommand{\gr}{\operatorname{gr}}
\newcommand{\Sp}{\operatorname{Sp}}\newcommand{\smp}{\mathfrak{sp}}
\newcommand{\cQ}{\mathcal{Q}}
\newcommand{\tcQ}{\tilde\cQ}
\newcommand{\hV}{\widehat V}
\newcommand{\cB}{\mathcal{B}}
\newcommand{\cM}{\mathcal{M}}
\newcommand{\cC}{\mathcal{C}}
\newcommand{\cN}{\mathcal{N}}

\newcommand{\Sec}{\operatorname{Sec}}
\newcommand{\Loc}{\operatorname{Loc}}
\newcommand{\MLoc}{\mathfrak{Loc}}
\newcommand{\Rees}{\operatorname{R}}
\newcommand{\Diff}{\cD}
\newcommand{\excise}[1]{}
\newcommand{\MDiff}{\mathfrak{Diff}}
\newcommand{\fS}{\mathfrak{S}}
\newcommand{\cE}{\mathcal{E}}
\newcommand{\cW}{\mathcal{W}}
\newcommand{\End}{\operatorname{End}}
\newcommand{\Aut}{\operatorname{Aut}}
\newcommand{\Res}{\operatorname{Res}}
\newcommand{\Ind}{\operatorname{Ind}}

\newcommand{\fM}{\mathfrak{M}}
\newcommand{\cZ}{\mathcal{Z}}
\newcommand{\fZ}{\mathfrak{Z}}
\newcommand{\fg}{\mathfrak{g}}
\newcommand{\mmod}{\operatorname{-mod}}
\newcommand{\pmmod}{\operatorname{-pmod}}
\newcommand{\red}{\mathfrak{r}}
\newcommand{\cOg}{\mathcal{O}_{\!\operatorname{g}}}
\newcommand{\dOg}{D_{\cOg}}
\newcommand{\preO}{p\cOg}
\newcommand{\dpreO}{D_{p\cOg}}
\newcommand{\HCg}{{}^{\mbox{}}\mathbf{HC}^{\operatorname{g}}}
\newcommand{\preHC}{p\HCg}
\newcommand{\alg}{T}
\newcommand{\si}{\sigma}
\newcommand{\talg}{\tilde{T}}
\newcommand{\bla}{{\underline{\boldsymbol{\la}}}}
\newcommand{\cOa}{\mathcal{O}_{\!\operatorname{a}}}
\newcommand{\Lotimes}{\overset{L}{\otimes}}
\newcommand{\thetitle}{Weighted Khovanov-Lauda-Rouquier algebras}

\usetikzlibrary{decorations.pathreplacing,backgrounds,decorations.markings,calc}
\tikzset{wei/.style={draw=red,double=red!40!white,double distance=1.5pt,thin}}

\renewcommand{\theitheorem}{\Alph{itheorem}}

\noindent {\Large \bf 
\thetitle}
\bigskip\\
{\bf Ben Webster}\footnote{Supported by the NSF under Grant DMS-1151473 and  by the NSA under Grant H98230-10-1-0199.}\\
Department of Mathematics,  University of Virginia, Charlottesville, VA
\bigskip\\
{\small
\begin{quote}
\noindent {\em Abstract.}
In this paper, we define a generalization of Khovanov-Lauda-Rouquier
algebras which we call {\bf weighted Khovanov-Lauda-Rouquier
  algebras}.  We show that these algebras carry many of the same
structures as the original Khovanov-Lauda-Rouquier algebras, including
induction and restriction functors which induce a twisted bialgebra
structure on their Gro\-then\-dieck groups.  

We also define natural {\bf steadied quotients} of these algebras,
which in an important special cases give categorical actions of an
associated Lie algebra.  These include the algebras categorifying
tensor products and Fock spaces defined by the author and Stroppel in
\cite{Webmerged,SWschur}.

For symmetric Cartan matrices, weighted KLR algebras also have a
natural geometric interpretation as convolution algebras, generalizing
that for the original KLR algebras by Varagnolo and Vasserot
\cite{VV}; this result has positivity consequences important in the
theory of crystal bases.   In this case, we can also relate the
Grothendieck group and its bialgebra structure to the Hall algebra of
the associated quiver.
\end{quote}
}
\bigskip

\section{Introduction}
\label{sec:introduction}

In this paper, we introduce a generalization of 
Khovanov-Lauda-Rouquier algebras \cite{KLI,Rou2KM}, which we call {\bf
  weighted Khovanov-Lauda-Rouquier algebras}.  The original KLR
algebras are finite dimensional algebras associated to a quiver, or
more generally a symmetrizable Cartan datum.  To define the weighted
generalization of these algebras, one must choose in addition a {\bf
  weighting} on the graph $\Gamma$ underlying the Cartan datum; this is simply
an assignment of a real number $\vartheta_e$ to each oriented edge of
$\Gamma$.  

This extra datum allows us to modify the relations of the KLR algebra in
a way which is simple, but will probably initially look strange even to
experts in the subject.  The essential paradigm shift is that instead
of beginning with idempotents indexed by sequences of nodes from the
Dynkin diagram $\Gamma$, one should assign an idempotent to a sequence
enriched with a position on the real number
line for each element of the sequence, remembering the distance between
points.  We call such an object a {\bf
  loading}.  The elements of our algebra will be
linear combinations of diagrams much like those of the KLR algebra, but unlike
the original relations, interesting relations can occur when strands
come within a fixed distance of each other; we call this phenomenon
``action  at a distance.''

If there is a single node and no loops, then there are no changes and
we arrive at the nilHecke algebra exactly as in the KLR case.  Let us
consider the next easiest case, where $\Gamma$ is a $A_2$ Dynkin
diagram.  As in the original KLR algebra (in Rouquier's presentation from
\cite[\S 3.2]{Rou2KM},
or as described in \cite{Webmerged,CaLa}), one must choose a polynomial
$Q_{12}(u,v)=au+bv$ that describes the interaction of these two strands via the
relation
 \[ 
\begin{tikzpicture}[very thick,xscale=1.6]
 \draw (1,0) to[in=-90,out=90]  node[below, at start]{$1$} (1.5,1) to[in=-90,out=90] (1,2)
;
  \draw (1.5,0) to[in=-90,out=90] node[below, at start]{$2$} (1,1) to[in=-90,out=90] (1.5,2);
\node at (2,1) {=};
\node at (2.4,1) {$a$};
  \draw (2.7,0) -- node[below, at start]{$1$} node[midway,fill,inner
  sep=2.5pt,circle]{} (2.7,2) ;
  \draw (3.2,0) to node[below, at start]{$2$} (3.2,2);
\node at (3.6,1) {$+$};
\node at (3.9,1) {$b$};
  \draw (4.2,0) -- (4.2,2) node[below, at start]{$1$};
  \draw  (4.7,0)-- node[midway,fill,inner sep=2.5pt,circle]{} (4.7,2) node[below, at start]{$2$};
\end{tikzpicture}
\]
If the weighting on the unique edge $e$ is $k<0$, then we will see
this relation not when a strand labeled $1$ crosses one labeled $2$ and
then crosses back, but when it passes the line $k$ units left of the
strand labeled $2$ and crosses back.  In order to aid with visualizing
this, we draw a dashed line $k$ units left of each strand labeled
$2$.  We will refer to these dashed lines as {\bf ghosts} throughout
the paper; in general, we must draw one for each pair consisting of a
strand labeled with some node $k$, and an edge whose head is $k$.  In
this case, we will arrive at the relation:

 \[ 
\begin{tikzpicture}[very thick,xscale=1.6]
 \draw (1,0) to[in=-90,out=90]  node[below, at start]{$1$} (1.5,1) to[in=-90,out=90] (1,2)
;
  \draw[dashed] (1.5,0) to[in=-90,out=90] (1,1) to[in=-90,out=90] (1.5,2);
  \draw (2.5,0) to[in=-90,out=90]  node[below, at start]{$2$} (2,1) to[in=-90,out=90] (2.5,2);
\node at (3,1) {=};
\node at (3.4,1) {$a$};
  \draw (3.7,0) -- (3.7,2)  node[midway,fill,inner sep=2.5pt,circle]{} node[below, at start]{$1$}
 ;
  \draw[dashed] (4.2,0) to (4.2,2);
  \draw (5.2,0) -- (5.2,2) node[below, at start]{$2$};
\node at (5.6,1) {$+$};
\node at (5.9,1) {$b$};
  \draw (6.2,0) -- (6.2,2) node[below, at start]{$1$};
  \draw[dashed] (6.7,0)-- (6.7,2);
  \draw (7.7,0) --   node[midway,fill,inner sep=2.5pt,circle]{} (7.7,2) node[below, at start]{$2$};
\end{tikzpicture}
\]
This case produces no interesting new algebras: we can recover the
original KLR relations by shifting all strands with label $2$ to the
left by $k$ units.  In general, we can always find such a fix when
$\Gamma$ is a tree.  However, when the graph $\Gamma$ has cycles,
interesting new algebras can appear.  For example, for the Jordan
quiver and the dimension vector $(n)$, we arrive at the smash product $\K[S_n]\# \K[x_1,\cdots, x_n]$.

Many properties of the original KLR algebras carry over: the weighted
KLR algebra has a permutation type basis and a faithful representation
representation on a sum of polynomials.  Its category of
representations is endowed with monoidal and co-monoidal structures
given by induction and restriction, generalizing those structures for
the KLR algebra.  Furthermore, its Grothendieck
group has a twisted bialgebra structure (or alternatively, Hopf
structure for a particular braided monoidal category) induced by these
functors.

This definition was motivated in large part by the desire to unify
generalizations of the KLR algebras that have appeared in the author's
previous work.  In order to develop these, we associate to a quiver
$\Gamma$ and dominant weight $\la$ a new quiver $\Gamma_\la$, which we
call its {\bf Crawley-Boevey quiver} (see Section \ref{sec:examples}).
These quivers appear naturally in the theory of Nakajima quiver
varieties.  The weighted KLR algebras attached to any weighting have a
natural quotient we call their {\bf steadied quotient} (see Section
\ref{sec:steadied-quotients}); these generalize the cyclotomic
quotients of usual KLR algebras and always carry a categorical
representation of the Kac-Moody algebra $\fg$ (see Theorem
\ref{th:categorical-action}).

These allow us to interpret the
tensor product algebras $T^\bla$ and $\tilde{T}^\bla$ defined in
\cite[\S 4]{Webmerged} and the (extended) quiver Schur algebras $A,
A^\bla$ and $\tilde{A}^\bla$ from \cite[\S 2\& \!4]{SWschur} in terms of
a single construction.
\begin{itheorem}
For each Cartan datum, and list of dominant weights
$\bla=(\la_1,\dots,\la_\ell)$, there is a weighting on the
Crawley-Boevey quiver of $\la=\la_1+\cdots +\la_\ell$ whose weighted
KLR algebra $\walg^\vartheta_\nu$ is isomorphic to $\tilde{T}^\bla_{\la-\nu}\otimes_\K\K[t]$.
The steadied quotient $\walg^\vartheta_\nu(c)$ of this KLR algebra is isomorphic to ${T}^\bla_{\la-\nu}\otimes_\K\K[t]$.

For $\Gamma$ a cycle, the weighted KLR algebra $\walg^\vartheta_\nu$ is either Morita
equivalent to the original KLR algebra or to the quiver Schur algebra
$A_\nu$, depending on whether the sum of weights on an oriented cycle is zero
or not.  In this case, there is also a weighting on the Crawley-Boevey
quiver for $\la$ and a fixed set of loadings  whose weighted
KLR algebra is Morita equivalent to
$\tilde{A}^\bla_{\la-\nu}\otimes_\K\K[t]$ with steadied quotient Morita
equivalent to ${A}^\bla_{\la-\nu}\otimes_\K\K[t]$.
\end{itheorem}
Another significant motivation is that more general steadied quotients
in the affine case are equivalent to category $\cO$ for a rational
Cherednik algebra of the group $G(r,1,\ell)$, as we prove in
\cite{WebRou}.  Numerous constructions from this paper, including
steadied quotients and canonical deformations play a key role in that work.

While this construction is purely algebraic in nature, it has a geometric inspiration: for a quiver
$\Gamma$ with vertex set $I$ and a dimension vector $d\colon \Gamma\to \Z_{\geq 0}$, an
integral weighting $\vartheta$ will define
a $\C^*$-action on \[E_\Gamma=\bigoplus_{i\to
  j}\Hom(\C^{d_i},\C^{d_j})\] by letting $t\cdot
(f_e)=(t^{\vartheta_e}f_e)$.  Varagnolo and Vasserot \cite{VV} have given an
interpretation of some KLR algebras as Ext-algebras of complexes of
constructible sheaves on the moduli stack $E_\nu/G_\nu$ of representations of the
quiver $\Gamma$ which appeared in work of Lusztig \cite{Lus91}; we can
generalize this construction to give an analogous constructible complex $Y$ of 
sheaves which is well-behaved with respect to the $\C^*$-action.
\begin{itheorem}
  The weighted KLR $\walg^\vartheta_\nu$ associated to a quiver
  $\Gamma$ with integral weighting is the Ext algebra $\Ext_{E_\nu/G_\nu}(Y,Y)$.  If
  $\operatorname{char}(\K)=0$ then $Y$ is semi-simple. 

  The map sending the class of a projective module $[P]$ to an
  appropriate Frobenius trace of $Y\otimes_{\walg^\vartheta_\nu}P$ on the
  $\mathbb{F}_p$ points of $E_\nu$ is a bialgebra map from
  $K^0_q(\walg^\vartheta_\nu)$ to the Hall algebra of the quiver $\Gamma$. 
\end{itheorem}
This theorem has
important positivity consequences; it is a key step in matching the
bases defined by projective objects with their canonical
bases in the sense of Lusztig (see \cite[\S 6]{WebCB} and \cite[\S
4.7]{WebRou}).  It will also play an important role in understanding generalizations
of category $\cO$ in forthcoming work on the representation theory of quantizations of
quiver varieties \cite{Webqui}.

\section{Basic properties}
\label{sec:basic}

\subsection{Weighted algebras defined}
\label{sec:weight-algebr-defin}

Consider a graph $\Gamma$ with vertex set $I$ and oriented edge set $\Omega$; we allow these edges to have multiplicities $c_e, c_{\bar e}\in \Z_{\geq 0}$ for $e\in \Omega$.
Let $h,t\colon \Omega\cup \bar\Omega\to I$ be the head and tail maps. 
 We assume these multiplicities are symmetrizable, in the sense that
 there exist $d_i$ such that $d_{h(e)}c_e=d_{t(e)}c_{\bar e}$. 

There are two important examples to keep in mind:
\begin{itemize}
\item If $C$ is a symmetrizable generalized Cartan matrix, then we
  have the associated Dynkin diagram $\Gamma$, with the multiplicities
  $c_e$ given by the negative of
  the entries $-c_{ij}$ of the Cartan matrix.  More
  generally, if $\fg$ has no loops, then there is an associated
  symmetrizable Kac-Moody algebra.
\item We can also take {\it any} locally finite graph $\Gamma$ with all $c_e=c_{\bar e}=1$.
\end{itemize}

Throughout, we will let a {\bf weighting} on a quiver mean simply a
map  $\vartheta\colon \Omega \to \R$; that is an attachment of a real
number to each edge.  By convention, we extend this function to $\bar
\Omega$ by $\vartheta_{\bar e}=-\vartheta_e$.  Note that we can also
think of this an $\R$-valued 1-cocycle on the underlying CW complex of $\Gamma$.

Fix a commutative ring $\K$.  For each edge, we choose a polynomial
$Q_e(u,v)\in \K[u,v]$ which is homogeneous
of degree $d_{h(e)}c_e=d_{t(e)}c_{\bar e}$ when
$u$ is given degree $d_{h(e)}$ and $v$ degree $d_{t(e)}$.   We will always
assume that $Q_e$ has coefficients before the pure monomials in $u$ and $v$ which are units, and 
set $Q_{\bar e}(u,v)=Q_e(v,u)$.  In particular, if $(\Gamma, c_*)$
arises from a symmetrizable Cartan matrix, the polynomials $Q_{ij}=Q_e$ satisfy the
properties we desire to define a KLR algebra (as in
\cite[\S 2]{Webmerged}). Furthermore, we assume that if $e$ is a loop of degree
0, then $Q_{e}(u,v)=(u-v)P_e(u,v)$ for some symmetric polynomial $P_e(u,v)$.

\begin{definition}
 A {\bf loading} $\Bi$ is a function from $\R$ to $I\cup\{0\}$ which is
  only non-zero at finitely many points.  We can also think a loading as choosing a finite subset of the real line and labeling its elements with simple roots.

A loading is called {\bf generic} if there is no real number such that $\Bi(a)=t(e), \Bi(a-\vartheta_e)=h(e)$ for some edge $e\in \Omega$, or such that $\Bi(a-\vartheta_e)=h(e), \Bi(a-\vartheta_{e'})=h(e')$ and $\vartheta_e\neq \vartheta_{e'}$. 
\end{definition}
If we think of our loading as a set of labeled points, we can
visualize this as adding a ``ghost'' of each point labeled $h(e)$ for
each edge $e\in \Omega$ which is $\vartheta_e$ units to the right of
the point, and require that none of these coincide with each other or with
points of the loading when it can be avoided. We let $|\Bi|=\sum_{r\in
  \R}\Bi(a)$, and let $d$ be the number of points in $\Bi$.

\begin{remark}
  The reader familiar with KLR algebras will be used to thinking of
  $\Bi$ as a sequence of simple roots which has an order, but no
  distance information.  From now on, the distance between these
  elements will be essential, in a way that will be clear
  momentarily.  We can always obtain a simple ordered list of nodes
  $\becircled{\Bi}$ by forgetting the positions of the points; we call
  this the {\bf unloading} of $\Bi$.
\end{remark}

Assume for now that
\begin{itemize}
\item[$(\dagger)$] $\Gamma$ is a graph such that no two edges of the same
  weight have matching tail and head, and there are no cyclically oriented bigons with opposite weights.
\end{itemize}
We now define the {\bf weighted KLR algebra} $\walg^\vartheta_B$ on a finite set of
loadings $B$.
\begin{definition}\label{diagram-def}
  A {\bf weighted KLR} (wKLR) {\bf diagram} is a collection of
  finitely many oriented smooth curves in
  $\R\times [0,1]$ with each oriented in the negative direction.  That
  is, each curve's projection to the $y$-coordinate must be a
  diffeomorphism to $[0,1]$. Each curve must have
one endpoint on $y=0$ and one on $y=1$, at distinct points from the other curves.
Curves are allowed to carry a finite number of dots.

Furthermore, for 
every edge with $h(e)=i$ we add a ``ghost'' of each strand labeled $i$
shifted $\vartheta_e$ units to the right (or left if $\vartheta_e$ is
negative).  We require that there are no tangencies or triple
intersection points between any combination of strands and ghosts, and
no dots on intersection points.  Note that by our assumption $(\dagger)$, at a generic horizontal slice of the diagram, no two ghosts, two strands, or pair of ghost and strand coincide, except for those strands and ghosts that coincide because of edges of weight 0.

We'll consider these diagrams up to isotopy which preserves all these conditions.
\end{definition}
For example, if we have an edge $i\to j$, then the diagram $a$ is a
wKLR diagram, whereas $b$ is not since it has a tangency between a
strand and a ghost: \[a=
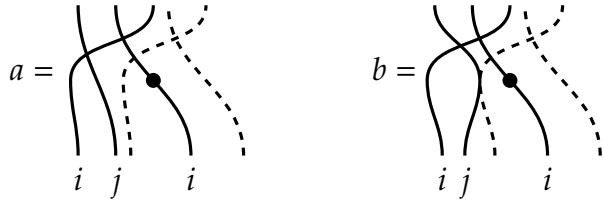
\begin{tikzpicture}[baseline,very thick]
  \draw (-.5,-1) to[out=90,in=-90] node[below,at start]{$i$} (-.6,0) to[out=90,in=-90](.5,1);
   \draw  (1,-1) to[out=90,in=-90] node[below,at start]{$i$}
  node[midway,circle,fill=black,inner sep=2pt]{} (0,1);
  \draw[dashed] (.2,-1) to[out=90,in=-90] (.1,0) to[out=90,in=-90](1.2,1);
   \draw[dashed]  (1.7,-1) to[out=90,in=-90]   (.7,1);
  \draw(0,-1) to[out=90,in=-90]node[below,at start]{$j$} (-.5,1);
\end{tikzpicture}\qquad \qquad 
b=
\begin{tikzpicture}[baseline,very thick]
  \draw (-.4,-1) to[out=90,in=-90] node[below,at start]{$i$} (-.6,0) to[out=90,in=-90](.5,1);
   \draw  (1,-1) to[out=90,in=-90] node[below,at start]{$i$}
  node[midway,circle,fill=black,inner sep=2pt]{} (0,1);
  \draw[dashed] (.3,-1) to[out=90,in=-90] (.1,0) to[out=90,in=-90](1.2,1);
   \draw[dashed]  (1.7,-1) to[out=90,in=-90]   (.7,1);
  \draw(-.1,-1) to[out=90,in=-90] node[below,at start]{$j$}(.1,0) to[out=90,in=-90] (-.5,1);
\end{tikzpicture}
\]

Reading along the lines $y=0,1$, we obtain loadings, which we call the
{\bf top} and {\bf bottom} of the diagram.   There is a notion of {\bf
  composition}  $ab$ of wKLR diagrams $a$ and $b$: this is
  given by stacking $a$ on top of $b$ and attempting to join the
  bottom of $a$ and top of $b$. If the loadings
  from the bottom of $a$ and top of $b$ don't match, then the
  composition is not defined and by convention is 0, which is not a
  wKLR diagram, just a formal symbol.
This composition rule makes the formal
span of all wKLR diagrams over $\K$ into an algebra
{\it $\doubletilde{W}^\vartheta$.}  For any finite set $B$ of
loadings, we let {\it
  $\doubletilde{W}^\vartheta_B$} be the subalgebra where we fix the
top and bottom of the diagram to lie in the set $B$.  For each loading
$\Bi\in B$, we have a straight line diagram $e_{\Bi}$ where every
horizontal slice is $\Bi$, and there are no dots.

We can define a {\bf degree} function on KL diagrams.  The degrees are
given on elementary diagrams by 
\begin{multline*}
  \deg\tikz[baseline,very thick,scale=1.5]{\draw (.2,.3) -- (-.2,-.1)
    node[at end,below,
    scale=.8]{$i$}; \draw (.2,-.1) -- (-.2,.3) node[at
    start,below,scale=.8]{$j$};}
  =-\delta_{i,j}\langle\al_i,\al_i\rangle \qquad
  \deg\tikz[baseline,very thick,scale=1.5]{\draw (0,.3) -- (0,-.1)
    node[at
    end,below,scale=.8]{$i$} node[midway,circle,fill=black,inner
    sep=2pt]{};}=\langle\al_i,\al_i\rangle\\
  \deg\tikz[baseline,very thick,scale=1.5]{\draw (.2,.3) -- (-.2,-.1)
    node[at
    end,below,scale=.8]{$i$}; \draw[dashed] (.2,-.1) -- (-.2,.3)
    node[at
    start,below,scale=.8]{$j$};} = \deg\tikz[baseline,very
  thick,scale=1.5]{\draw[dashed] (.2,.3) -- (-.2,-.1) node[at
    end,below,scale=.8]{$i$}; \draw (.2,-.1) -- (-.2,.3) node[at
    start,below,scale=.8]{$j$};}
  =-\frac{1}{2}\langle\al_i,\al_j\rangle (1-\delta_{i,j})
\end{multline*}

For a general diagram, we sum together the degrees of the elementary
diagrams it is constructed from.  

  \begin{definition}\label{def-wKLR}
    The {\bf weighted KLR algebra} $\walg^\vartheta_B$ is the quotient
    of {\it
  $\doubletilde{W}^\vartheta_B$} by relations similar to the original
KLR relations, but with interactions between differently labelled strands
turned into relations between strands and ghosts of others. 
 If there
is a loop of weight $0$ at $i$ (there can be at most one), we let
$P_i(u,v)$ be the polynomial $Q_e(u,v)/(u-v)$ attached to this loop earlier; if there
is no such loop, we let $P_i(u,v)=0$.

We give the list of local relations below. 
Some care must be used when understanding what it means to apply these
relations locally.  In each case, the LHS and RHS have a dominant term
which are related to each other via an isotopy through a disallowed
diagram with a tangency, triple point or a dot on a crossing.  You can only apply the
relations if this isotopy avoids tangencies, triple points and dots on crossings
everywhere else in the diagram; one can always choose isotopy
representatives sufficiently generic for this to hold.
\begin{enumerate}
\item The relations for passing dots through crossings are exactly as in the KLR algebra. 
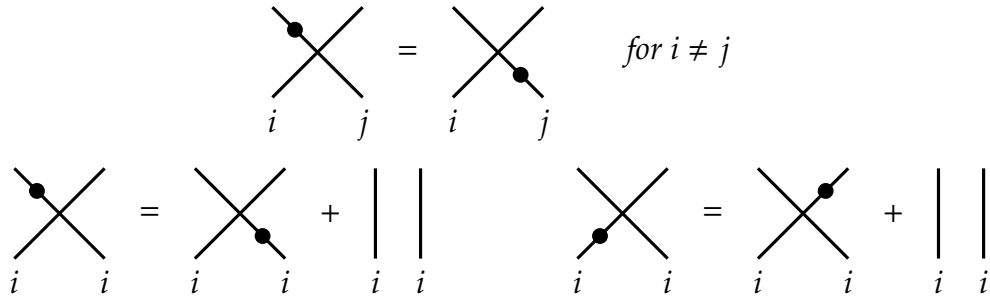
\begin{equation*}
    \begin{tikzpicture}[scale=.6,baseline]
      \draw[very thick](-4,0) +(-1,-1) -- +(1,1) node[below,at start]
      {$i$}; \draw[very thick](-4,0) +(1,-1) -- +(-1,1) node[below,at
      start] {$j$}; \fill (-4.5,.5) circle (5pt);
      \node at (-2,0){=}; \draw[very thick](0,0) +(-1,-1) -- +(1,1)
      node[below,at start] {$i$}; \draw[very thick](0,0) +(1,-1) --
      +(-1,1) node[below,at start] {$j$}; \fill (.5,-.5) circle (5pt);
      \node at (4,0){for $i\neq j$};
    \end{tikzpicture}
\end{equation*}
\begin{equation*}
    \begin{tikzpicture}[scale=.6,baseline]
      \draw[very thick](-4,0) +(-1,-1) -- +(1,1) node[below,at start]
      {$i$}; \draw[very thick](-4,0) +(1,-1) -- +(-1,1) node[below,at
      start] {$i$}; \fill (-4.5,.5) circle (5pt);
      \node at (-2,0){=}; \draw[very thick](0,0) +(-1,-1) -- +(1,1)
      node[below,at start] {$i$}; \draw[very thick](0,0) +(1,-1) --
      +(-1,1) node[below,at start] {$i$}; \fill (.5,-.5) circle (5pt);
      \node at (2,0){+}; \draw[very thick](4,0) +(-1,-1) -- +(-1,1)
      node[below,at start] {$i$}; \draw[very thick](4,0) +(0,-1) --
      +(0,1) node[below,at start] {$i$};
    \end{tikzpicture}\qquad \qquad
    \begin{tikzpicture}[scale=.6,baseline]
      \draw[very thick](-4,0) +(-1,-1) -- +(1,1) node[below,at start]
      {$i$}; \draw[very thick](-4,0) +(1,-1) -- +(-1,1) node[below,at
      start] {$i$}; \fill (-4.5,-.5) circle (5pt);
      \node at (-2,0){=}; \draw[very thick](0,0) +(-1,-1) -- +(1,1)
      node[below,at start] {$i$}; \draw[very thick](0,0) +(1,-1) --
      +(-1,1) node[below,at start] {$i$}; \fill (.5,.5) circle (5pt);
      \node at (2,0){+}; \draw[very thick](4,0) +(-1,-1) -- +(-1,1)
      node[below,at start] {$i$}; \draw[very thick](4,0) +(0,-1) --
      +(0,1) node[below,at start] {$i$};
    \end{tikzpicture}
  \end{equation*}
  \item If we undo a bigon formed by the $m$th strand and the ghost of
    the $n$th coming from the edge $e$ (assuming $e$ is not a loop with
    $\vartheta_e= 0$), then we separate the strands and multiply by
    $Q_e(y_n,y_m)$.  This is a bit harder to draw in complete
    generality, but for example, if there is an edge $e\colon i\to j$
    with $\vartheta_e<0$ 
    and $Q_e(u,v)=au+bv$, then we have 
 
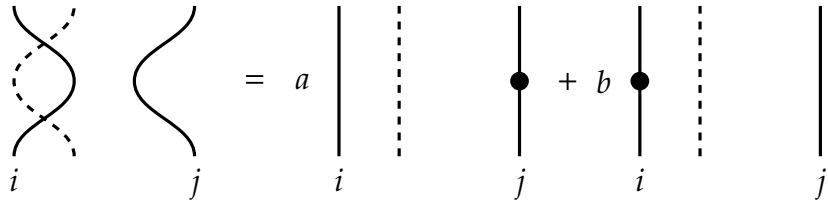
\[ 
\begin{tikzpicture}[very thick,xscale=1.6]
 \draw (1,0) to[in=-90,out=90]  node[below, at start]{$i$} (1.5,1) to[in=-90,out=90] (1,2)
;
  \draw[dashed] (1.5,0) to[in=-90,out=90] (1,1) to[in=-90,out=90] (1.5,2);
  \draw (2.5,0) to[in=-90,out=90]  node[below, at start]{$j$} (2,1) to[in=-90,out=90] (2.5,2);
\node at (3,1) {=};
\node at (3.4,1) {$a$};
  \draw (3.7,0) -- (3.7,2) node[below, at start]{$i$}
 ;
  \draw[dashed] (4.2,0) to (4.2,2);
  \draw (5.2,0) -- (5.2,2) node[below, at start]{$j$} node[midway,fill,inner sep=2.5pt,circle]{};
\node at (5.6,1) {$+$};
\node at (5.9,1) {$b$};
  \draw (6.2,0) -- (6.2,2) node[below, at start]{$i$} node[midway,fill,inner sep=2.5pt,circle]{};
  \draw[dashed] (6.7,0)-- (6.7,2);
  \draw (7.7,0) -- (7.7,2) node[below, at start]{$j$};
\end{tikzpicture}
\]
  \item If we undo a bigon formed by the $k$th strand and the $k+1$st
    strand, we simply separate the strands if they have different
    labels.  If they are both labelled with $i$, then then the result
    is a single crossing of the strands times $2P_i(y_k,y_{k+1})$.
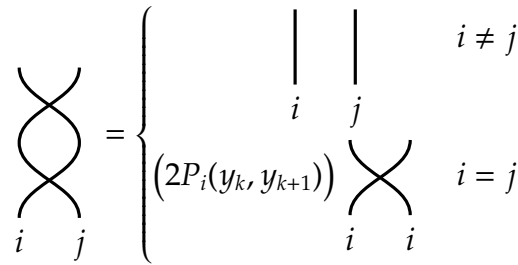
\[  \begin{tikzpicture}[very thick,xscale=1.6,baseline]
 \draw (1,-1) to[in=-90,out=90]  node[below, at start]{$i$} (1.5,0) to[in=-90,out=90] (1,1)
;
  \draw (1.5,-1) to[in=-90,out=90] node[below, at start]{$j$} (1,0) to[in=-90,out=90] (1.5,1);
\end{tikzpicture}=
\begin{cases}\qquad \qquad
   \begin{tikzpicture}[very thick,xscale=1.6,baseline]
 \draw (1,-.5) to[in=-90,out=90]  node[below, at start]{$i$} (1,.5)
;
  \draw (1.5,-.5) to[in=-90,out=90] node[below, at start]{$j$} (1.5,.5);
\end{tikzpicture} & i\neq j\\ \big(2 P_i(y_k,y_{k+1}) \big)
 \begin{tikzpicture}[very thick,xscale=1.6,baseline=-3pt]
 \draw (1,-.5) to[in=-90,out=90]  node[below, at start]{$i$} (1.5,.5)
;
  \draw (1.5,-.5) to[in=-90,out=90] node[below, at start]{$i$} (1,.5);
\end{tikzpicture}& i=j
\end{cases}
\]
\item strands can move through triple points without effect, except
  \begin{enumerate}
  \item when a ghost for an edge $e\colon i\to j$ which is $\vartheta_e$ to the right of the $m$th strand
    (which is labelled $j$) passes
    through a crossing of the $n$th and $n+1$st strands and these both have
label $i$. In  this case the diagrams where the
    strand is at the left differs from the one where it is at the right
    by 
\[
\partial_{n,n+1}Q_{e}(y_m,y_{n})=\frac{Q_e(y_m,y_n)-Q_e(y_m,y_{n+1})}{y_n-y_{n+1}}.\]
  \item the $m$th strand (which is labelled $i$) passes through the
    ghosts attached to $e\colon i\to j$ attached to the
    of the $n$th and $n+1$st strands, which are both labelled $j$.
 In  this case the diagrams where the
    strand is at the left differs from the one where it is at the right
    by 
\[
\partial_{n,n+1}Q_{e}(y_{n},y_m)=\frac{Q_e(y_n,y_m)-Q_e(y_{n+1},y_m)}{y_n-y_{n+1}}.\]
As before, we will not try to draw a completely general picture, but
given an example when there is an edge $e\colon i\to j$, $\vartheta_e<0$ and
$Q_e(u,v)=au+bv$, then we have 
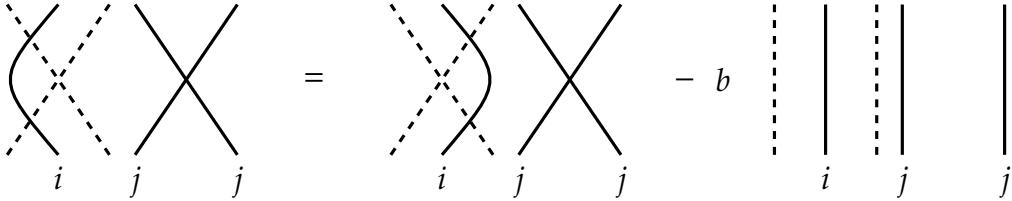
\begin{equation*}
    \begin{tikzpicture}[very thick,xscale=1.7]
      \draw[dashed] (-3,0) +(.4,-1) -- +(-.4,1);
 \draw[dashed]      (-3,0) +(-.4,-1) -- +(.4,1); 
    \draw (-2,0) +(.4,-1) -- +(-.4,1) node[below,at start]{$j$}; \draw
      (-2,0) +(-.4,-1) -- +(.4,1) node[below,at start]{$j$}; 
 \draw (-3,0) +(0,-1) .. controls (-3.5,0) ..  +(0,1) node[below,at
      start]{$i$};\node at (-1,0) {=};  \draw[dashed] (0,0) +(.4,-1) -- +(-.4,1);
 \draw[dashed]      (0,0) +(-.4,-1) -- +(.4,1); 
    \draw (1,0) +(.4,-1) -- +(-.4,1) node[below,at start]{$j$}; \draw
      (1,0) +(-.4,-1) -- +(.4,1) node[below,at start]{$j$}; 
 \draw (0,0) +(0,-1) .. controls (.5,0) ..  +(0,1) node[below,at
      start]{$i$};
\node at (1.9,0)
      {$-$};   
\node at (2.2,0) {$b$};
     \draw (4,0)
      +(.4,-1) -- +(.4,1) node[below,at start]{$j$}; \draw (4,0)
      +(-.4,-1) -- +(-.4,1) node[below,at start]{$j$}; 
 \draw[dashed] (3,0)
      +(.4,-1) -- +(.4,1); \draw[dashed] (3,0)
      +(-.4,-1) -- +(-.4,1); 
\draw (3,0)
      +(0,-1) -- +(0,1) node[below,at start]{$i$};
    \end{tikzpicture}
  \end{equation*}
\item the triple point involves the $m$th, $m+1$st and $m+2$nd
  strands, all labelled $i$ and there is a loop of weight $0$ joining
  $i$ to itself. In  this case the diagrams where the
    strand is at the left differs from the one where it is at the right
    by  \begin{multline*} \big(P_i(y_k,y_{k+1})P_i(y_{k+1},y_{k+2})+P_i(y_k,y_{k+2})P_i(y_{k+1},y_k)-P_i(y_k,y_{k+2})P_i(y_{k+1},y_{k+2})\big)\psi_k\\-\big(P_i(y_k,y_{k+1})P_i(y_{k+1},y_{k+2})+P_i(y_k,y_{k+2})P_i(y_{k+2},y_{k+1})-P_i(y_k,y_{k+2})P_i(y_{k},y_{k+1})\big)\psi_{k+1}\label{eq:1}
\end{multline*}
\end{enumerate} 

\end{enumerate}
  \end{definition}

\begin{proposition}\label{orientation-reverse}
  If we reverse the orientation of an edge $e\mapsto e'$, and set
  $\vartheta_{e'}'=-\vartheta_e$ and $Q_{e'}'(u,v)=Q_e(v,u)$, then
  $\walg^{\vartheta'}\cong \walg^{\vartheta}$ via the obvious isomorphism
  leaving strands unchanged.   
\end{proposition} 
By analogy with the geometry of Section \ref{sec:geometry-quivers}, we
call this isomorphism {\bf Fourier transform}.

\begin{definition} \label{def:arbitrary}
  If $\Gamma$ is an arbitrary choice of graph with multiplicities, $\vartheta_e$ and 
  $Q_e$ associated polynomials, then the {\bf weighted KLR algebra}
  $\walg^\vartheta_B$ for a set of loadings $B$ is the weight KLR
  algebra for the graph where we replace all bigons where the weights match (perhaps after reversing the orientation and negating the weight) with single edges of that weight, with $Q_{\operatorname{new}}=\prod Q_{\operatorname{old}}$.  Proposition \ref{orientation-reverse} shows that this does not depend on how one chooses to reverse orientations.
\end{definition}
We note that this algebra has a natural anti-automorphism where $a^*$
is the reflection of a diagram $a$ through a horizontal line.

Of course, many readers used to more categorical language will prefer
to think that there is a category where the objects are loadings, and
the morphism spaces are the spaces $e_\Bi \walg^\vartheta_B e_{\Bj}$
described above.  We will freely switch between these two formalisms
throughout the paper.

\subsection{A permutation type basis}
\label{sec:perm-type-basis}

\begin{proposition}\label{prop:action}
  This algebra $W^\vartheta_B$ acts on a sum of polynomial
  rings $\oplus_{B}\K[y_1,\dots, y_d]$, one for each loading, via the rule 
  \begin{itemize} 
\item when a strand passes from right of a ghost to
    left, we take the identity. 
\item when the $j$th strand passes
    from left of the ghost for $e$ of the $k$th strand to right of it,
    we multiply by $Q_e(y_k,y_j)$. 
\item when the $j$ and $j+1$
    strands cross and have the different labels, we just apply the
    permutation $s_j$. 
\item when the $j$ and $j+1$ strands cross and
    have the same label $i$, we act with the Demazure operator $ \partial_{j,j+1}=\frac{s_j-1}{y_{j+1}-y_{j}}$ if there is no loop of weight 0 at $i$ and if there is such a loop $e$, we act by 
    $Q_e(y_j,y_{j+1})\cdot \partial_{j,j+1}=P_e(y_j,y_{j+1})\cdot(1-s_j).$
  \end{itemize} \end{proposition} 
\begin{proof} The confirmation of
  the relations is an easy modification of the proof of Khovanov and
  Lauda \cite{KLI}. The relations (1) follow from the usual Leibnitz
  rule for Demazure operators:
  \begin{equation}
    \label{eq:demazure}
    \partial_{j,j+1}(fg)=\frac{f^{s_i}g^{s_i} -fg}{y_{j+1}-y_{j}}=f^{s_i}\partial_{j,j+1}(g)+\partial_{j,j+1}(f)g.
  \end{equation}
The relation (2) is simply follows from the fact that one of the
crossings introduces a factor of $Q_e(y_k,y_j)$, and the other a
factor of $1$.  The relation (3) is just $s_k^2=1$ if $i\neq j$, and
if $i=j$, then for the no loop case, this is just $\partial_k ^2=0$
and in the case where there is a loop, we have
\[P_i(y_k,y_{k+1})(1-s_k) P_i(y_k,y_{k+1})(1-s_k) =P_i(y_k,y_{k+1})^2
(1-s_k)^2=2 P_i(y_k,y_{k+1})^2 (1-s_k).\]
The relations  (4a) and (4b) follows immediately from \eqref{eq:demazure}.

The only really different relation to check is (4c); in this case, we
  use the notation $P_{ij}=P_e(y_{k+i-1},y_{k+k-1})$. The action we
  check is \begin{multline*} \tikz[very thick,scale=.5,baseline]{\draw
      (-1,1) to [out=-45,in=90] (1,-1); \draw (-1,-1) to
      [in=-90,out=45] (1,1); \draw (0,-1) to [out=135,in=-90] (-.7,0)
      to[out=90,in=-135] (0,1); }= P_{12}\circ (s_k-1) \circ
    P_{23}\circ (s_{k+1}-1) \circ P_{12}\circ (s_k-1)
= P_{12}P_{13}P_{23}s_ks_{k+1}s_k-  P_{12}P_{13}P_{23}s_ks_{k+1}\\- P_{12}P_{13}P_{23}s_{k+1}s_{k}+P_{12}P_{13}P_{23}s_{k+1} +(P_{12}P_{23}+P_{21}P_{13})s_k - (P_{12}P_{23}+P_{21}P_{13})P_{12}
 \end{multline*}
Comparing with the mirror image, we arrive at the desired relations.
\end{proof}
Fix a pair of loadings $\Bi,\Bj$.  For each permutation $\pi$ such that the order of labels appearing in the loadings $\Bi,\Bj$ differ by $\pi$, we fix an diagram $b_\pi$ which wires together $\Bi$ and $\Bj$ according to that permutation.  

Note that now even for a transposition of adjacent elements, this is not uniquely determined, since we may have a ghost that passes between both the pairs of elements which we wire in opposite order, and the element depends on whether we cross our strands to the left or right of this ghost; we let $\psi_k$ denote the diagram in which we cross to the left of all possible ghosts. Obviously, these generate the algebra together with the dots $y_i$.
\begin{theorem}\label{th:basis}
 The space $e_{\Bi}\walg^\vartheta e_{\Bj}$ is a free module over $\K[y_1,\dots, y_m]$, and the diagrams $b_\pi$ are a free basis.
\end{theorem}
\begin{proof}
Proof that these span is much like that of \cite[Lemma 4.11]{Webmerged}.  If
the strands of a diagram ever cross each other twice, or cross a ghost
twice, we can rewrite them as a sum of diagrams
with fewer crossings between pairs of strands or strands and ghosts
using the relations of Definition \ref{def-wKLR}(2-4).
Thus, we need only consider diagrams that we could have chosen for
$b_\pi$.  Furthermore, we can use the triple-point moves to show that the
difference between any two such diagrams for $\pi$ has fewer crossings
by Definition \ref{def-wKLR}(4).
  Thus, the $b_\pi$'s must span and we
  need only show they are linearly independent.    

On the rational
  functions in the polynomial
  representation, the element $b_\pi$ acts as a product of operators
  which are of the form $s_i$ times a rational function plus a
  rational function times 1.  The operator $s_i$ commutes past multiplying by
  a rational function just by acting on it (the smash product rule);
  thus the product of these terms is $\pi$ times a rational function,
  plus of a sum of shorter elements of $S_n$ times rational
  functions.  Thus, the linear independence over $\K[y_1,\dots, y_m]$
  of the action of the elements of $S_n$ guarantees the linear
  independence of the $b_\pi$'s.
\end{proof}
Note that in the course of this proof, we've also shown that the
action of Proposition \ref{prop:action} is faithful.

\subsection{Dependence on choice of loadings}
\label{sec:depend-choice-load}

\begin{definition}\label{def:equivalent}
    Call two loadings $\Bi,\Bi'$ {\bf equivalent} if for every edge
    $e:i\to j$, and each pair of integers $(f,g)$ the ghost of the $f$th strand labeled with $h(e)$ is either to the left of the $g$th strand labeled $t(e)$ in both $\Bi, \Bi'$ or to the right in both.
\end{definition}

 \begin{example}
  Let $\Gamma$ be the Kronecker quiver 
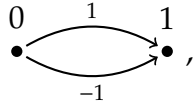
\[\tikz[thick,
   baseline=-3.5pt]{\node[fill=black,circle,inner sep=1.5pt,outer
     sep=2pt,label=above:$0$] (a) at (0,0){}; \node[fill=black,circle,inner
     sep=1.5pt,outer sep=2pt,label=above:$1$] (b) at (2,0){}; \draw[->] (a)
     to[out=30,in=150] node [above,midway,scale=.7]{$1$} (b); \draw[->] (a) to[out=-30,in=-150] node [below,midway,scale=.7]{$-1$} (b);},\]
   with the two edges are given weights $1$ and $-1$. For
   $\nu=\al_0+\al_1$, a loading is determined the $x$-coordinates
   $x_0$ and $x_1$ of the points labeled with $0$ and $1$.  There are
   3 equivalence classes of loadings determined by the
   inequalities
   \begin{equation}
\tikz[baseline]{\node at (-5,0) [label=above:{$x_0<x_1-1$}]{\tikz[baseline=-2pt,very thick, xscale=2] { 
\draw (1.25,-.5) -- node[below, at start]{$1$} (1.25,.5);
\draw[dashed] (.65,-.5) -- (.65,.5);
\draw[dashed] (-.1,-.5) -- (-.1,.5);
\draw (.5,-.5) -- node[below, at start]{$0$} (.5,.5);
 }};
\node at (0,0) [label=above:{$x_1-1< x_0 < x_1+1$}]{ \tikz[baseline=-2pt,very thick, xscale=2] { 
\draw (.05,-.5) -- node[below, at start]{$1$} (0.05,.5);
\draw[dashed] (-.55,-.5) -- (-.55,.5);
\draw[dashed] (-.1,-.5) -- (-.1,.5);
\draw (.5,-.5) -- node[below, at start]{$0$} (.5,.5);}};
 \node at (5,0) [label=above:{$x_1+1<x_0.$}]{\tikz[baseline=-2pt,very thick, xscale=2] { 
\draw (1.25,-.5) -- node[below, at start]{$0$} (1.25,.5);
\draw[dashed] (.65,-.5) -- (.65,.5);
\draw[dashed] (-.1,-.5) -- (-.1,.5);
\draw (.5,-.5) -- node[below, at start]{$1$} (.5,.5);}};
 }\label{eq:3}
\end{equation}

 \end{example}

\begin{proposition}
In the algebra on any set $B$ of loadings containing equivalent
loadings $\Bi,\Bi'\in B$, the projective modules $\walg^\vartheta
e_{\Bi}$ and $\walg^\vartheta e_{\Bi'}$ are isomorphic.  That is, the
original algebra  is Morita equivalent to that with either loading
excluded.

 In terms of the category of loadings mentioned earlier, these
 loadings are isomorphic.
\end{proposition}

\begin{proof}
  The straight-line path from $\Bi$ to $\Bi'$ gives an isomorphism between these projectives.
\end{proof}
In particular, if we simultaneously translate all points in a loading, we will obtain an equivalent one. 

  Consider the dominant cone $D_n=\{x_1< \cdots <x_n\}\subset
  \R^n$. For each $\nu=\sum v_i\al_i$, the set of loadings with
  $|\Bi|=\nu$ is naturally identified with the product of the dominant
  cones $D_{v_1}\times \cdots \times D_{v_m}\subset \R^{v_1}\times
  \cdots \times \R^{v_m}$ minus finitely many affine hyperplanes.
  It's clear from the definition that:

  \begin{proposition}
    The sets of equivalence classes are precisely the connected components
    of the complement in  $D_{v_1}\times \cdots \times D_{v_m}$ of affine hyperplanes associated to each edge
    $e\colon i\to j$ and $1\leq m\leq v_i, 1\leq n\leq v_j$: 
\[H_{e,m,n}=\{x_m^{(i)}-x_n^{(j)}=\vartheta_e\}.\] 
  \end{proposition}

 In particular, there are only finitely many equivalence classes for each fixed $\nu$.  
  \begin{definition}
Let $B(\nu)$ denote a fixed choice of a set of loadings containing one
from each equivalence class with $|\Bi|=\nu$.

    From now on, when we say ``the weighted KLR algebra'' $\walg^\vartheta_\nu$ we mean using that attached to the set $B(\nu)$ of loadings; this algebra is unique up to canonical isomorphism, and if we add any new generic loadings with $|\Bi|=\nu$ to this algebra, we will always obtain a Morita equivalent algebra.  Generally, we will not carefully distinguish between equivalent loadings and will freely replace inconvenient loadings with equivalent ones.
  \end{definition}
In terms of the category of loadings, we have simply chosen a set of
objects such that any object is isomorphic to one of the collection;
this is almost the skeleton of the category, but we have not accounted
for the fact that sometimes non-equivalent loadings will be isomorphic.
Thus, the weighted KLR algebra can be thought of really as an
equivalence class of linear categories, and from this perspective, it
is manifestly well-defined.

 For simplicity, we fix a real number $s>|\vartheta_e|$ for all $e$.
 Let $B_s$ be the set of loadings where the points of the
 loading are spaced exactly $s$
 units apart and the first point is at $x=0$.  Such loadings are in
 canonical bijection with sequences of elements in $I$.  For the Kronecker quiver weighted as in the example above, we must have $s>1$, and only
 the first and third loadings of \eqref{eq:3} are included in $B_s$.

\begin{proposition}
  If the graph $\Gamma$ has no loops, then the algebra $\walg^\vartheta_{B_s}$ is isomorphic to the original KLR algebra, with \[Q_{ij}(u,v)=\prod_{\substack{i=h(e)\\j=t(e)}}Q_{e}(u,v).\]  

In particular, if $\vartheta_e=0$ for all $e$, we obtain the usual KLR algebra.
\end{proposition}
\begin{proof}
  This isomorphism matches $e_\Bi$ to an idempotent in the KLR algebra
  for the corresponding sequence in $I$; the dot $y_k$ and crossing
  $\psi_k$ correspond to the similarly named elements as well.
  Our condition on loadings forces that (after ``pulling taut'') the $j$th strand crosses the
  $k$th if and only if it crosses all its ghosts; the relations
  induced between such crossings are exactly the original KLR relations.
\end{proof}

This does not fully exhaust the cases where actually only obtain the original algebra.  This is
easier to see once we consider a symmetry of our definition.  We can view the weighting $\vartheta$ as a 1-chain on $\Gamma$.  If $\eta\colon I\to \R$ is a 0-chain, then we can consider the cohomologous 1-chain $(\vartheta+d\eta)_e=\vartheta_e+\eta_{h(e)}-\eta_{t(e)}$.
\begin{proposition}
 The map $\walg^{\vartheta}_B\to \walg^{\vartheta+d\eta}_B$ moving each $i$-labelled strand $\eta_i$ units right is an isomorphism.
\end{proposition}
\begin{proof}
  This map moves the ghost attached to an edge $e$ to the right by $\eta_{t(e)}$, so this map maintains all crossings between strands of the same color and between ghosts and strands labelled with the tail of the associated edges.
\end{proof}
\begin{corollary}\label{cor:tree}
  If $\Gamma$ is a tree, $\walg^\vartheta_\nu$ is Morita equivalent to the original KLR algebra.
\end{corollary}
Note that we say ``Morita equivalent'' here, since the set
${B_s}$ may actually contain redundant loadings which are equivalent
to each other (since equivalence is insensitive to the relative
ordering of nodes with no edge connecting them).

\subsection{Induction and restriction}
\label{sec:induct-restr}

For each decomposition $\nu=\nu'+\nu''$, we have a map
$\iota_{\nu';\nu''}\colon \walg^\vartheta_{\nu'}\otimes
\walg^\vartheta_{\nu''}\to \walg^\vartheta_\nu$, where we send a tensor
product of diagrams $a\otimes b$ to the diagram where they are placed
next to each other with $s$ units of separation between them.  Note
that this map is not unital, but sends $1\otimes 1$ to an idempotent
$e_{\nu';\nu''}$.  Up to the isomorphism induced by changing a loading
in its equivalence class, this isomorphism is unchanged by adjusting
the distance between the diagrams, as long as it is sufficiently
large.  This can be thought of as an induction operation on loadings
themselves: $\iota_{\nu';\nu''}(e_{\Bi}\otimes e_{\Bj})=e_{\Bi\circ \Bj}$.

\begin{definition}
  Define the functor of {\bf induction} by \[\Ind^{\nu}_{\nu';\nu''}(M,N)=M\circ N:=\walg^\vartheta_\nu\otimes_{\walg^\vartheta_{\nu'}\otimes \walg^\vartheta_{\nu''}}M\boxtimes N\]
and {\bf restriction} by 
\[\Res^{\nu}_{\nu';\nu''}(L):=e_{\nu';\nu''}L.\]
\end{definition}

\begin{proposition}
  The operation $\circ$ makes the sum $\oplus_\nu
  \walg^\vartheta_\nu\mmod$ into a monoidal category, and $\Res_{*,*}$
  makes this sum into a comonoidal category.  The subcategory
  $\oplus_\nu
  R_\nu\mmod$ is monoidally generated by
  $\walg^\vartheta_{\al_i}\mmod$.
\end{proposition}
Recall that the Grothendieck group $ K^0(\walg^\vartheta_\nu)$ is
the span of formal symbols corresponding to finitely generated
projective $\walg^\vartheta_\nu$-modules subject to the relation that
$[M\oplus N]= [M]+[N]$; we can think of the sum $K=\oplus_\nu
K^0(\walg^\vartheta_\nu)$ as an abelian group graded by $\Z[I]$.
Furthermore, we endow $\Z[I]$ with a pairing where \[i\cdot
j=2d_i\delta_{ij}-d_{i}\big(\sum_{j\overset{e}{\to} i} c_e+\sum_{i\overset{e}{\to}
  j}c_{\bar e}\big)=2d_j\delta_{ij}-d_j \big(\sum_{j\overset{e}{\to} i} c_{\bar
  e}+\sum_{i\overset{e}{\to} j}c_{ e}\big).\]  We will sometimes view
this as the symmetrization of the bilinear form \[\langle
j,i\rangle=d_i\delta_{ij}-\sum_{j\overset{e}{\to} i} d_ic_e \qquad i\cdot j=\langle
i,j\rangle+\langle
j,i\rangle.\]
This allows us to define a twisted product structure on $A\otimes A$
for any $\Z[I]$-graded algebra $A$ by $q^{\deg(b)\cdot
  \deg(c)}(a\otimes b)(c\otimes d)$. As noted by Walker \cite{Walker}, we can
think of this as the natural product in the braided 
monoidal category of $\Z[I]$-graded vector spaces, where the braiding
map on a tensor product of spaces $V$ of pure degree $\mu$ and $V'$ of
degree $\mu'$ is the switch map $V\otimes V'\to V'\otimes V$ times
$q^{\mu\cdot \mu'}$.
\begin{theorem}
  The Grothendieck group $K=\oplus_\nu K^0(\walg^\vartheta_\nu)$ endowed
  with \[\text{the product } [M][N]=[M\circ N]\text{ and coproduct }
  \Delta([L])=\sum_{\nu'+\nu''=\nu}[\Res^{\nu}_{\nu';\nu''}(L)]\] is a
  twisted bialgebra with a natural map $U^+_q(\fg_\Gamma)\to K$; in
  fact, it is a Hopf algebra in the braided category of $\Z[I]$-graded vector spaces.
\end{theorem}
\begin{proof}
For a decomposition $\nu=\nu_1+\nu_2=\nu_1'+\nu_2'$, we consider the
  restriction of $\walg^\vartheta_\nu$ to $\walg^\vartheta_{\nu_1}\otimes
  \walg^\vartheta_{\nu_2}$ on the left and $\walg^\vartheta_{\nu_1'}\otimes
  \walg^\vartheta_{\nu_2'}$ on the right.  We can filter $\walg^\vartheta_\nu$
  as a bimodule by the sum  $\mu$ of the labels on the strands that pass from
  left to right, so the sum of the labels passing right to left
  is $\mu'=\nu_1'-\nu_1+\mu$. By the same argument as
  \cite[2.18]{KLI}, the successive quotients of this
  filtration are \[(\walg^\vartheta_{\nu_1-\mu;\mu}\otimes
  \walg^\vartheta_{\mu';\nu_2-\mu'})\otimes
  _{\walg^\vartheta_{\nu_1-\mu}\otimes \walg^\vartheta_{\mu}\otimes
  \walg^\vartheta_{\mu'}\otimes \walg^\vartheta_{\nu_2-\mu'}}(\walg^\vartheta_{\nu_1-\mu;\mu'}\otimes
  \walg^\vartheta_{\mu;\nu_2-\mu'})\] shifted upwards by the inner product
  $-\langle \mu,\mu'\rangle$.  As noted in \cite[3.2]{KLI}, this suffices to
  prove that the coproduct $\Delta$ is an algebra map $K\to K\otimes
  K$ for the twisted product structure.  

The counit $\epsilon$ just kills $K^0(\walg^\vartheta_\nu)$ for $\nu\neq
0$, and the antipode $S$, as in the work of Xiao \cite{Xiao}, can be constructed inductively by the formula
\[S([M])=-\sum_{\substack{\nu=\nu'+\nu''\\ \nu''\neq 0}}(1\otimes S)[\Res^{\nu}_{\nu';\nu''}(M)]\qedhere\]
\end{proof}

\subsection{The twisted algebra}
\label{sec:twisted-algebra}

There is a larger category $\mathcal{P}$ whose objects are pairs $(\Bi;\vartheta)$
of loadings and weights.  Morphisms $(\Bi_0;\vartheta_0)$ and
$(\Bi_1;\vartheta_1)$ between two such pairs is very
much like in the category of loadings for a fixed weight, but the distance from each
ghosts to the strand 
it haunts is not a constant: instead at the horizontal slice $y=a$, the distance of
a ghost for $e:i\to j$ from the corresponding $j$ labeled strand is $a\vartheta_1(e)+(1-a)\vartheta_0(e)$.  All the same local relations between
morphisms apply without change.  

\begin{proposition}
This category has a representation that associates a polynomial ring
to each pair $(\Bi;\vartheta)$ with the action given by formulas as in
Proposition \ref{prop:action}.
  The morphism space between any two pairs in $\mathcal{P}$ is spanned
  by a basis given by the product of monomials in the dots with a
  fixed stringing up of each permutation.
\end{proposition}
\begin{proof} We can define an action on a sum of polynomial rings by the same local rules as
  \ref{prop:action}; since the same local relations are used, the same
  proof carries through.  With this action in hand, we can use the
  same proof as Theorem
  \ref{th:basis}.
\end{proof}

We will often be interested in considering the sum of all morphism
spaces from loadings with one weighting $\vartheta$ to those with
another $\vartheta'$.  This sum is naturally a bimodule
$B^{\vartheta,\vartheta'}$ over $\walg^\vartheta$ and $\walg^{\vartheta'}$.

\subsection{Steadied quotients}
\label{sec:steadied-quotients}

In this subsection, we define a natural quotient of
$\walg^\vartheta_\nu$; while the algebraic 
motivation for this definition may not be immediately apparent, we believe it is well-motivated both by examples and
by geometry.  In fact, we recommend that the reader glance at the next
section on examples before reading the definition below. 

A {\bf charge} on the vertex set $I$ is a map $c\colon I\to \C_+$ where
\[\C_+=\{x\in \C\mid \text{either } \Im(x)>0 \text{ or } x\in
\R_{>0}.\}\] We always extend $c$ linearly to $\Z[I]$.
Such a charge induces a preorder $>_c$ on $\Z_{>0}[I]$, using the argument
of $c(\bd)$ 

\begin{definition}
  We call an indecomposable $\walg^\vartheta_\nu$-module is called {\bf unsteady} if it is
  isomorphic to a  summand of an induction $M_1\circ M_2$ where
  $\wt(M_1)>_c\wt(M_2)$.  
\end{definition}

In $\walg^\vartheta_\nu$, there is a natural 2-sided ideal $I_c$ generated by
all elements factoring through unsteady projectives (thought of as a
map of left modules $\walg^\vartheta_\nu\to \walg^\vartheta_\nu$).  Visually,
this corresponds to diagrams where in the middle of the diagram, there
is a horizontal slice whose the induced loading is $\Bi_1\circ\Bi_2$
where $|   \Bi_1|>_c|\Bi_2|$.

\begin{definition}
  The {\bf steadied} quotient $\walg^\vartheta_\nu(c)$ of
  $\walg^\vartheta_\nu$ is the quotient  $\walg^\vartheta_\nu/I_c$.  We let
 $B^{\vartheta,\vartheta'}(c)$ denote the compatible quotient of the
 bimodule $B^{\vartheta,\vartheta'}$.
\end{definition}

\subsection{Canonical deformations}
\label{sec:canon-deform}

The algebras $\walg^\vartheta$ have a canonical deformation.  For each
edge $e$ with head $j$ and tail $i$, we assign an alphabet of
variables $z_{e,a,b}$ for integers $0\leq a < d_i, 0\leq b < d_j$ such that
$ad_j+bd_i< d_jc_e=d_ic_{\bar e}$.  We then consider the weighted
KLR algebra over the ring $\K[\mathbf{z}_e]$ with $Q$-polynomials
given by \[\tQ_{e}(u,v)=Q_{e}(u,v)+\sum_{a,b}z_{e,a,b}u^av^b.\]  This
polynomial will be homogeneous if we endow $z_{e,a,b}$ with degree
$d_jc_e -ad_j-bd_i =d_ic_{\bar e} -ad_j-bd_i$.  Let $S=\K[\{z_{e,a,b}\}]$. In
the case where each edge has multiplicity 1 ($c_e=1$), then we only
have one variable per edge and $\tQ_{e}(u,v)=Q_{e}(u,v)+z_{e}$.  

\begin{proposition}
  This deformation is free (and thus flat) over $S$.
\end{proposition}
\begin{proof}
The proof of Theorem \ref{th:basis} works equally well over $S$,
showing that the diagrams $b_\pi$ give a free basis over $S[y_1,\dots,
y_m]$.  By multiplying by monomials, we easily obtain a free $S$-basis.
\end{proof}

Fix a field $K$, and a non-zero homomorphism $\chi\colon S\to
K$.  Fix a finite subset $M_i$ of $K$ for each $i\in I$.
\begin{definition}
  The graph $\Gamma_{\chi,M_\bullet}$ is the graph with underlying set
  $ \bigcup_{i\in I} \{i\}\times M_i\subset \Gamma\times K$. For
$q_1\in M_i,q_2\in M_j$, an edge $e\colon i\to j$ lifts to an edge $\check{e}$
from $(i,q_1)$ and
$(j,q_2)$ if and only if the polynomial satisfies
$\chi(\tilde{Q}_{e})(q_1,q_2)=0$.
\end{definition}
Note that the 
natural map $\sqcup_{i\in I} M\to \Gamma$ is a graph homomorphism.
We can naturally assign polynomials to this graph by 
\[Q_{\check{e}}(u,v):=\chi(\tilde{Q}_{ij})(u+q_1,v+q_2).\] Given a
weighting $\vartheta$ of $\Gamma$, we also 
weight $\Gamma_{\chi,M_\bullet}$ with $\check{\vartheta}_{\check{e}}=\vartheta_e$.  

\begin{example}
  Assume $\Gamma$ is an $e$-cycle, whose vertices we identify with
  $\Z/e\Z=\{0,\dots, e-1\}$ with an edge $i\to i+1$.  If send $z_e$ for the edge
  $e\colon e-1\to 0$
  to $-1$ and set $z_e$ for every other edge 
  of this graph to $0$, with $K$ any characteristic 0 field, and take
  $M=\Z$.  We thus find that we have an edge $(p,q) \to (p',q')$ if
  $p'\equiv p+1\mod e$ and $q'-q=\delta_{p',0}$.  This is equivalent
  to $q'e+p'=qe+p+1$. That is, the resulting graph $\Gamma_{\chi,M_\bullet}$ is isomorphic to
  $\Z$ with an edge $i\to i+1$, where we identify $\Z$ and $\Z\times
  \Z/e\Z$ by division with remainder by $e$.  
\end{example}
\begin{example}
  Let $\Gamma$ be any graph, and let $K$ any field, with each $z_e$
  sent to 0.   For any finite subset $M\subset K$, we can set
  $M_i=M$.  The resulting graph is just $\Gamma\times M$, with the map
  to $\Gamma$ being a trivial $\#M$-fold covering.
\end{example}
\begin{example}
  If $\Gamma$ has a non-symmetric Cartan matrix, then for each pair
  $i,j\in I$, we let
  $e_{ij}=\gcd(c_{ij},c_{ji}),
  f_{ij}=c_{ij}/\gcd(c_{ij},c_{ji})$.
  Consider the polynomials
  $Q_e(u,v)=(u^{f_{ij}}-v^{f_{ji}})^{e_{ij}}$, let $K$ be a field of
  characteristic coprime to each $d_i$, and let $M_i$ be the
  $p_i=\operatorname{lcm}(\{d_k\})/d_i$th roots of unity in $K$.  In this case, the
  graph structure is that $\zeta_1$ and $\zeta_2$ are connected by an
  edge if $\zeta_1^{f_{ij}}=\zeta_2^{f_{ji}}$.  That is, each preimage
  of $i$ is connected to preimages of $j$ by $p_j/f_{ji}=\operatorname{lcm}(\{d_k\})
  \gcd(c_{ij},c_{ji})/d_jc_{ji}$ preimages, along edges with
  multiplicity $e_{ij}$.  

Thus, $\Gamma_{\chi,M_\bullet}$ in this case is the standard branched
cover of a non-symmetric Cartan matrix by a symmetric one.  
\end{example}

We'd like to understand the specialization $\walg_\nu^\vartheta\otimes_SK$ at the
homomorphism $\chi$; while we don't have a general description of this
algebra, we can consider a natural completion of it.  

Let $I_k\subset
\walg_\nu\otimes_SK$ be the two-sided ideal in $\walg_\nu\otimes_SK$
generated by the products $\prod_{m\in 
  M}(y_i-m)^{k}$ for each $i$. 
 These are clearly nested, and
have trivial intersection for reasons of degree;  thus, we can
consider the completion $\widehat{\walg_\nu^\vartheta\otimes_SK}$ at this system
of ideals.  Note, that this depends in a very strong way on $M$, but
we will suppress this dependence from the notation.  On the other hand,
we can consider the weighted KLR algebra $\widehat{\dwalg^{\check{\vartheta}}}$  of the
graph $\Gamma \times M$ over the field $K$,  completed by the
two-sided ideals generated by $y_i^{k}$ for all $i$.
This is the same completion applied before, but with
$M=\{0\}$.  

The completion $\widehat{\walg_\nu^\vartheta\otimes_SK}$ has a natural
decomposition according to the topological generalized eigenvalues of the
operators $y_i$. That this, we can decompose each quotient
$\walg_\nu^\vartheta\otimes_SK/I_n$ according to these eigenvalues
since it is finite dimensional, and take the inverse limit of this
decomposition.  Note that these generalized eigenvalues must lie in
$M$, since the minimal polynomial of $y_i$ on
$\walg_\nu^\vartheta\otimes_SK/I_n$ divides $\prod_{m\in 
  M}(y_i-m)^{k}$. This decomposes the idempotents $e_{\Bi}$
corresponding to loadings as a sum of idempotents where we associate
an additional choice of $m\in M$ to each point in the loading.  Put
another way, consider the ways of lifting the loading in $\Gamma$ to
one in $\Gamma\times M$.  If $\check{\Bi}$ is such a loading, let
$\epsilon_{\check{\Bi}}$ denote the projection to its generalized
eigenspace (which is an element of the algebra by abstract Jordan
decomposition in each quotient).  

For any weighted KLR diagram for the graph $\Gamma\times
  M$, we have a ``projection'' where we apply the first projection to
  the labels of each strand; we can always isotope a KLR diagram so
  that this projection is a weighted KLR diagram as well (if we aren't careful,
  we might introduce tangencies).  Note that this result might not be
  independent of the isotopy.  

\begin{proposition}\label{completion}
There is an isomorphism $\widehat{\dwalg^{\check{\vartheta}}}\cong\widehat{\walg_\nu^\vartheta\otimes_SK}$ such that:
\begin{equation}
e_{\check{\Bi}}\mapsto \epsilon_{\check{\Bi}}\qquad 
y_i e_{\check{\Bi}}\mapsto  (y_i-m_i) \epsilon_{\check{\Bi}}\label{eq:4}
\end{equation}
For diagrams, it is easier to describe this map locally.  For most
diagrams with a single crossing and no dots, we simply pass to the projection, times
$\epsilon_{\check{\Bi}}$, except in cases where:
\begin{itemize}
\item At $y=a$, in the projection of $A$, there is a crossing where the $\ell$th strand (call its
    label $i$) crosses from left to right of
    a ghost haunting the $k$th strand for an edge $e\colon i\to
    j$ which doesn't lift to an edge $\check{e}\colon
    (i,m_\ell)\to (j,m_k)$. In this case, we multiply the projection at $y=a$
by   $\chi(\tilde{Q}_{e})(y_\ell+m_\ell,y_k+m_k)^{-1}$.  This exists because $\chi(\tilde{Q}_{e})(y_\ell+m_\ell,y_k+m_k)$ is a power
    series with non-zero constant term by assumption, and thus
    invertible.
\item At $y=a$, there is a crossing of two strands with labels
    $(i,m_k)$ and $(i,m_{k+1})$ with $m_k\neq m_{k+1}$.
We send the crossing to $y_{k+1}-y_{k}$ times the projection diagram plus the
diagram with the crossing opened.  That is:
\[ \tikz[baseline,scale=1.1,very thick]{\draw (-.5,-.5) -- node [below, scale=.8,at
  start]{$(i,m_k)$} (.5,.5);\draw (.5,-.5) -- node [below, scale=.8,at
  start]{$(i,m_{k+1})$} (-.5,.5);} \mapsto  \tikz[baseline,very thick]{\draw
  (-.5,-.5) --node [below, at
  start]{$i$} (.5,.5);\draw (.5,-.5) --  node[pos=.75,circle,fill=black,inner sep=2pt]{} node [below, at
  start]{$i$} (-.5,.5);} -\tikz[baseline,very thick]{\draw
  (-.5,-.5) -- node[pos=.75,circle,fill=black,inner sep=2pt]{} node [below, at
  start]{$i$} (.5,.5);\draw (.5,-.5) -- node [below, at
  start]{$i$} (-.5,.5);} + \tikz[baseline,very thick]{\draw (.5,-.5) -- node [below, at
  start]{$i$} (.5,.5);\draw (-.5,-.5) -- node [below, at
  start]{$i$} (-.5,.5);} \]
\end{itemize}
\end{proposition}
\begin{proof}
  Much like in \cite{WebBKnote}, we identify these algebras by giving
  an isomorphism between their completed polynomial representations.

  The completion of the polynomial representation of
  $\widehat{\walg_\nu^\vartheta\otimes_SK}$ is a sum of completed
  polynomial rings $\oplus_{\check{\Bi}} K[[y_1-m_1,\dots, y_n-m_n]]
  \epsilon_{\check{\Bi}}$, so we can use \eqref{eq:4} as the
  definition of the isomorphism of this to the completed polynomial
  representation of $\widehat{\dwalg^{\check{\vartheta}}}$.  

Thus, we
  need only check that dotless diagrams act correctly.  In all the
  cases where a diagram is sent to its projection, the match between
  the actions is clear.  Now consider the case where there is a crossing where the $\ell$th strand (call its
    label $i$) crosses from left to right of
    a ghost for the $k$th strand and an edge $e\colon i\to
    j$ which doesn't lift to an edge $\check{e}\colon
    (i,m_\ell)\to (j,m_k)$; in this case, the action of the projection
    is by
    multiplication by $\chi(\tilde{Q}_{e})(y_\ell+m_\ell,y_k+m_k)$.
    Thus, $\chi(\tilde{Q}_{e})(y_\ell+m_\ell,y_k+m_k)^{-1}$ times this
    diagram acts by the identity map, as does the diagram for
    $\Gamma\times M$.

    Finally, consider the case where there is a crossing of two strands with labels
    $(i,m_k)$ and $(i,m_{k+1})$ with $m_k\neq m_{k+1}$.  The
    projection acts by the Demazure operator
    $\frac{s_j-1}{y_{j}-y_{j+1}}$.  Thus, \[ \tikz[baseline,very thick]{\draw
  (-.5,-.5) --node [below, at
  start]{$i$} (.5,.5);\draw (.5,-.5) --  node[pos=.75,circle,fill=black,inner sep=2pt]{} node [below, at
  start]{$i$} (-.5,.5);} -  \tikz[baseline,very thick]{\draw
  (-.5,-.5) -- node[pos=.75,circle,fill=black,inner sep=2pt]{} node [below, at
  start]{$i$} (.5,.5);\draw (.5,-.5) -- node [below, at
  start]{$i$} (-.5,.5);} + \tikz[baseline,very thick]{\draw (.5,-.5) -- node [below, at
  start]{$i$} (.5,.5);\draw (-.5,-.5) -- node [below, at
  start]{$i$} (-.5,.5);} \] acts by the switch map $s_j$, as does the
diagram for $\Gamma\times M$.

Thus we need only check that this map is invertible.  The inverse
applied to a diagram times $\epsilon_{\check{\Bi}}$ similarly goes to
the ``anti-projection'' but times
$\chi(\tilde{Q}_{e})(y_\ell+m_\ell,y_k+m_k)$ where there is an
appropriate crossing of a strand and a ghost, and when two
like-colored strands with different $m_{k}$ and $m_{k+1}$ cross, the inverse map
is given by
\[ \epsilon_{\check{\Bi}}\tikz[baseline=-5pt,scale=1.1,very thick]{\draw (-.5,-.5) -- node [below, at
  start]{$i$} (.5,.5);\draw (.5,-.5) -- node [below, at
  start]{$i$} (-.5,.5);} \mapsto (y_k-y_{k+1}+m_k-m_{k+1})^{-1}\Bigg(\tikz[baseline=-5pt,very thick]{\draw
  (-.5,-.5) --node [below, scale=.8,at
  start]{$(i,m_k)$} (.5,.5);\draw (.5,-.5) --  node [below, scale=.8,at
  start]{$(i,m_{k+1})$} (-.5,.5);} - \tikz[baseline=-5pt,very thick]{\draw (.5,-.5) -- node [below, scale=.8,at
  start]{$(i,m_k)$} (.5,.5);\draw (-.5,-.5) -- node [below, scale=.8,at
  start]{$(i,m_{k+1})$} (-.5,.5);} \Bigg).\qedhere\]
\end{proof}


Note that this also induces a map on the level of steadied quotients,
since the loading $\hat{\Bi}$ is unsteady if and only if $\Bi$ is, and
the idempotent $\epsilon_{\hat{\Bi}}$ is 0 in the steadied quotient if
$\Bi$ is.


\section{Relation to previous constructions}
\label{sec:examples}

The motivation for the definition of weighted KLR algebras was to give
a unifying framework to some seemingly disparate examples, as well as
providing a language for new ones.

As Corollary \ref{cor:tree} shows, we will encounter nothing new if we
consider the weighted KLR algebras for a tree; in particular, for
any Dynkin diagram, or extended Dynkin diagram of type other than
$\widehat A_n$, nothing interesting happens.  On the other hand, there are
some very interesting cases based on slightly less famous graphs.

\subsection{The Crawley-Boevey trick and categorical actions}
\label{sec:Crawley-Boevey}

The most important case for us is the graph produced by ``the
Crawley-Boevey trick;'' this was a construction which was originally
designed with the aim of thinking of Nakajima's quiver varieties,
which were originally defined using auxilliary ``shadow vertices,'' as
a space of usual representations of a pre-projective algebras.  

Given a graph $\Gamma$ and a function $w\colon I\to \Z_{\geq 0}$, we can define a new graph $\Gamma_w$ where we take the original graph $\Gamma$, add a new vertex $0$ and string in $w_i$ edges from $0$ to $i$.  More formally, $\Gamma_w$ has vertex set $I\cup \{0\}$ and edge set $\Omega \cup \{e_i^{1},\dots, e_i^{(w_i)}\}_{i\in I}$ with $t(e_i^{(k)})=0,h(e_i^{(k)}
)=i$.  We call the original edges of $\Gamma$ {\bf old edges}, and the
edges $e_*^{(*)}$ {\bf new edges}.  For simplicity, we always choose $c_{e_i^{(w_i)}}=c_{\bar e_i^{(w_i)}}=1$ and $Q_{e_i^{(k)}}(u,v)=u-v$. 

As we noted, this graph has previously appeared in the literature on Nakajima quiver varieties, since 
\begin{itemize} 
    \item there's a canonical bijection between representations of
      $\Gamma_w$ with $V_0\cong k$ and representations of $\Gamma$ together with a choice of map
      $\C^{w_i}\to V_i$, and 
   \item similarly, representations of the preprojective algebra of
     $\Gamma_w$ with $V_0\cong k$ are in canonical bijection with elements of the vector space Nakajima denotes $\mathbf{M}$ subject to the moment map conditions \cite[(2.5)]{Nak94}, and 
\item this representation is stable in the
      sense of Craw for the character which is the product of the determinants of the action on $V_i$'s for $i\in I$, and the  $-\sum_{i\in I}\dim V_i$-power of the determinant on $V_0$ if and only if it is stable as in \cite[3.5]{Nak94}.
\end{itemize}

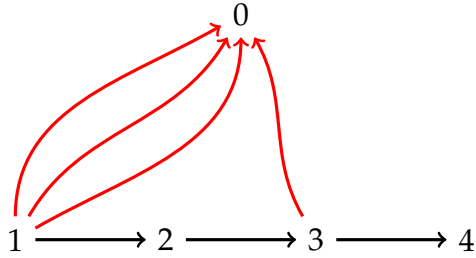
\begin{figure}[tbh]
  \centering
  \begin{tikzpicture}[very thick, ->, scale=2]
    \node (a) at (-2,0){$1$}; 
    \node (b) at (-1,0){$2$};
    \node (c) at (0,0){$3$};
    \node (d) at (1,0){$4$};
    \node (e) at (-.5,1.5){$0$};
\draw (a)--(b);
\draw (b)--(c);
\draw (c)--(d);
\draw[red] (a) to[out=90,in=-150] (e);
\draw[red] (a) to[out=60,in=-120] (e);
\draw[red] (c) to[out=120,in=-60] (e);
\draw[red] (a) to[out=30,in=-90] (e);
  \end{tikzpicture}
  \caption{The Crawley-Boevey quiver of $3\omega_1+\omega_3$ for $\mathfrak{sl}_5$.}
  \label{fig:crawley-boevey}
\end{figure}

This observation carries over into the algebras attached to these
quivers.  Given a highest weight $\la$ of the Kac-Moody Lie algebra
$\fg$ associated to $\Gamma$, we let $\Gamma_\la=\Gamma_w$ where $w(i)=\la(\al_i^\vee)$.

For any weighting $\vartheta$, call the {\bf reduced} quotient $\bar
\walg^\vartheta_{\tilde \nu}$ of the algebra $\walg^\vartheta_{\tilde \nu}$
for $\Gamma_w$ with weight $\tilde \nu=\nu+\al_0$ by the ideal
generated by all dots on the $0$-labelled strand. Consider the charge
$c$ which assigns $c(i)=-1+i$ for all old vertices and $i+\sum d_i$ to
$0$, and the reduced steadied quotient $\bar \walg^\vartheta_{\tilde
  \nu}(c)$.  When we relate this construction to the geometry of
quiver representations, this will
correspond to only acting by change of basis on the old vertices.  

Since the single strand with label $0$ in each diagram of $\bar
\walg^\vartheta_{\tilde \nu}$ plays a special role, we will represent
its ghosts using red ribbons like $\{\,\tikz[baseline]{\draw[wei]
  (0,-.05)--(0,.3);}\,\}$; this is suggestive of a relationship to the
tensor product algebras of \cite[\S 4]{Webmerged} which we will discuss shortly.

Assume that $\Gamma$ has no loops.  Recall that there is a 2-category
$\tU$, defined using the ring $\K$ and the polynomials $Q_{ij}$, which
categorifies the universal enveloping algebra of the associated
Kac-Moody algebra $\fg$.  We use the conventions established in our
previous papers \cite{Webcatq,Webmerged} for this category which (modulo
minor conventional differences) is that defined by Cautis and Lauda \cite{CaLa} building on work
of Rouquier \cite{Rou2KM} and Khovanov and Lauda \cite{KLIII}.  

\begin{theorem}\label{th:categorical-action}
  There is a categorical action of the Kac-Moody Lie algebra $\fg$ on
  the categories $\bigoplus_{\nu}\bar \walg^\vartheta_{\tilde
    \nu}(c)\pmmod$, with $\eF_i$ given by the induction functor
  $M\mapsto M\circ \walg^\vartheta_{\al_i}$, and $\eE_i$ by its left
  adjoint.  
\end{theorem}

 This is in principle the same proof as \cite[Thms. 4.25 \&
 4.28]{Webmerged}.  We define a ``doubled'' version of $
\bar \walg^\vartheta_{\tilde
    \nu}(c)$ analogous to the double cyclotomic quotient $DR^\la$,
  introduced in \cite[\S 3.1]{Webmerged}.  Much like $DR^\la$, the
  category of modules over  $
D\bar \walg^\vartheta_{\tilde
    \nu}(c)$ manifestly carries a $\tU$ action, but it is not {\it a
    priori} clear that it is ever non-zero.  However, we will prove
  that $D\bar \walg^\vartheta_{\tilde
    \nu}(c)$ and $\bar \walg^\vartheta_{\tilde
    \nu}(c)$ are Morita equivalent, allowing use to prove Theorem
  \ref{th:categorical-action}.

Consider the category $\mathcal{Y}'_\vartheta$ whose objects are {\bf signed loadings}, that is,
loadings where each point is marked with a $+$ or $-$, which we can
also represent as either an upward or downward arrow. We'll use
$i_{\pm}$ to represent the label of a point in a signed loading.

We let a {\bf blank double weighted KLR diagram} be a collection of curves
which are decorated with dots which are
oriented and match the up and down arrows on the source loading at
$y=0$ and the target at $y=1$, and are generic in the same sense as
weighted KLR diagrams.  These strands have ghosts
positioned $\vartheta_e$ units right of each strand (regardless of
orientation) labelled with the head of $e$; for purposes of weight
labeling we also need to include ghosts for the opposite orientation,
that is ghosts (which we will draw as dotted lines $\tikz[baseline=-3pt]{\draw[thick,dotted] (0,0)
  -- (.5,0); }$) $\vartheta_e$ units left of each strand labelled with
$t(e)$.  The diagrams are the same as those used in the
2-category $\tU$, except for the genericity conditions imposed by
ghosts.   Here is an example of such a diagram: 

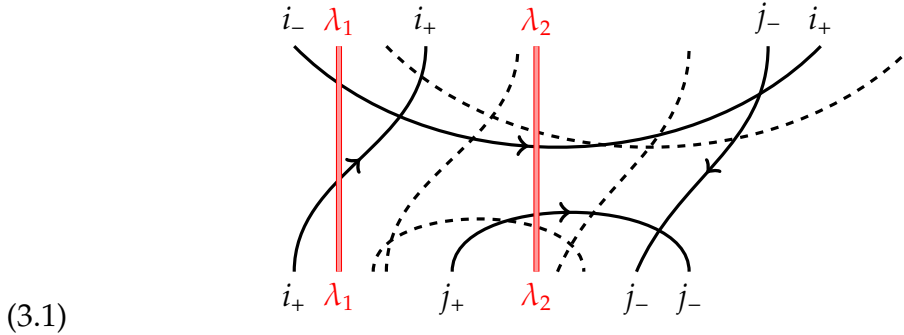
\begin{equation}
  \begin{tikzpicture}[very thick,yscale=3,xscale=3.5]
    \draw[postaction={decorate,decoration={markings,
    mark=at position .5 with {\arrow[scale=1.3]{>}}}}] (0,0) to[out=90,in=-90] node[at start, below]{$i_+$} node[at
    end, above]{$i_+$}  (.5,1); 
    \draw[postaction={decorate,decoration={markings,
    mark=at position .45 with {\arrow[scale=1.3]{>}}}}] (0,1) to[out=-50,in=-130] node[at start, above]{$i_-$} node[at
    end, above]{$i_+$}  (2,1) ; 
    \draw[postaction={decorate,decoration={markings,
    mark=at position .5 with {\arrow[scale=1.3]{>}}}}] (.6,0) to[out=90,in=90] node[at start, below]{$j_+$} node[at
    end, below]{$j_-$}(1.5,0);
    \draw[postaction={decorate,decoration={markings,
    mark=at position .5 with {\arrow[scale=1.3]{<}}}}] (1.3,0) to [out=70,in=-90] node[at start, below]{$j_-$} node[at
    end, above]{$j_-$} (1.8,1);
    \draw[dashed] (.35,0) to[out=90,in=-90] (.85,1); 
    \draw[dashed] (.35,1) to[out=-50,in=-130] (2.35,1); 
    \draw[dashed] (.3,0) to[out=90,in=90] (1.1,0);
    \draw[dashed] (1,0) to [out=70,in=-90] (1.5,1);
\draw[wei]  (.92,0) to node[at start, below,red]{$\la_2$} node[at end, above,red]{$\la_2$} (.92,1);
\draw[wei]  (.17,0) to node[at start, below,red]{$\la_1$} node[at end, above,red]{$\la_1$} (.17,1);
  \end{tikzpicture}\label{fig:Y-diagram}  
\end{equation}
Some care is necessary when labeling the regions of the plane.  We let
a {\bf double weighted KLR diagram} be a blank DWKLRD with a labeling
of each region of the plane minus strands and ghosts labeled by a
weight of $\fg$.  Rather
than using the rules of \cite{KLIII} or \cite{Webmerged}, these must
be consistent with the rules\footnote{When the Cartan matrix is not
  invertible, we should be a bit careful about precisely what
  fundamental weights mean, but this is actually a red herring.  What we
  really want to assign to regions are functions $I\to \Z$, but it has
been conventionally handy to write these functions in the form
$\al_i^\vee(\mu)$ for some weight $\mu$.  Thus, pedants should
consider $\om_i$ to be the characteristic function of $i\in I$.} that 
\begin{equation*}
  \tikz[baseline,very thick]{
\draw[wei,postaction={decorate,decoration={markings,
    mark=at position .5 with {\arrow[scale=1.3]{<}}}}] (0,-.5) -- node[below,at start]{$\la$}  (0,.5);
\node at (-1,0) {$\mu$};
\node at (1,.05) {$\mu+\la$};
}\qquad \qquad 
  \tikz[baseline,very thick]{
\draw[postaction={decorate,decoration={markings,
    mark=at position .5 with {\arrow[scale=1.3]{<}}}}] (0,-.5) -- node[below,at start]{$i$}  (0,.5);
\node at (-1,0) {$\mu$};
\node at (1,.05) {$\mu-2\om_i$};
}
\end{equation*}
and for ghosts corresponding to an edge $e\colon i\to j$:
\begin{equation*}
  \tikz[baseline,very thick]{
\draw[postaction={decorate,decoration={markings,
    mark=at position .5 with {\arrow[scale=1.3]{<}}}}, dashed] (0,-.5) -- node[below,at start]{$e$}  (0,.5);
\node at (-1,0) {$\mu$};
\node at (1,.05) {$\mu+c_{\bar{e}}\om_i$};
}\qquad \qquad 
  \tikz[baseline,very thick]{
\draw[postaction={decorate,decoration={markings,
    mark=at position .5 with {\arrow[scale=1.3]{<}}}}, dotted] (0,-.5) -- node[below,at start]{$e$}  (0,.5);
\node at (-1,0) {$\mu$};
\node at (1,.05) {$\mu+c_{e}\om_j$};
}
\end{equation*}
As in \cite[\S 2]{Webmerged}, we let $\EuScript{L}$ denote the label of the
leftmost region, and similarly for $\EuScript{R}$ and the rightmost.  We refine the scalars
$t_{ij}=Q_{ij}(1,0)$ as follows: for an edge $e$ and node $i$, we let
\[ t_{i;e}=
\begin{cases}
  Q_e(1,0)& \text{if $i=h(e)$,}\\
1 & \text{otherwise.} 
\end{cases}
\qquad u_{i;e}=\begin{cases}
  Q_e(0,1)& \text{if $i=h(e)$,}\\
1 & \text{otherwise.} 
\end{cases}\]

We let $\mathcal{Y}_\theta$ be the 2-category with:
\begin{itemize}
\item objects given by weights of $\fg$.
\item 1-morphisms $\la\to \mu$ given by loadings with label $\EuScript{L}=\la,\EuScript{R} =\mu$.  Composition is the horizontal
  composition of loadings.  
\item 2-morphisms $\Bi\to \Bj$ given by double weighted KLR diagrams
  with $\Bi$ as bottom and $\Bj$ as top, modulo
  the relations \cite[(2.2-4)]{Webmerged}, the adjunction, infinite
  Grassmannian and bigon relations corresponding to Lauda's categorification of $\mathfrak{sl}_2$ and
  \begin{itemize}
  \item the bigon relation for differently color strands \cite [(2.5a-b)]{Webmerged} is replaced by 
\newseq
\begin{equation*}\subeqn
\label{through-strands}
\begin{tikzpicture}[baseline=8pt]
    \node at (0,0){\begin{tikzpicture} [scale=1.3] \node[scale=1.5] at
        (-.7,1){$\la$};
        \draw[postaction={decorate,decoration={markings, mark=at
            position .5 with {\arrow[scale=1.3]{<}}}},very thick]
        (0,0) to[out=90,in=-90] node[at start,below]{$i$} (1,1)
        to[out=90,in=-90] (0,2) ;
        \draw[postaction={decorate,decoration={markings, mark=at
            position .5 with {\arrow[scale=1.3]{>}}}},very thick]
        (1,0) to[out=90,in=-90] node[at start,below]{$j$} (0,1)
        to[out=90,in=-90] (1,2);
      \end{tikzpicture}};

    \node at (1.7,0) {$=$}; \node at (3.2,0){\begin{tikzpicture}
        [scale=1.3]

        \node[scale=1.5] at (2.4,1){$\la$};

        \draw[postaction={decorate,decoration={markings, mark=at
            position .5 with {\arrow[scale=1.3]{<}}}},very thick]
        (1,0) to[out=90,in=-90] node[at start,below]{$i$} (1,2);
        \draw[postaction={decorate,decoration={markings, mark=at
            position .5 with {\arrow[scale=1.3]{>}}}},very thick]
        (1.7,0) to[out=90,in=-90] node[at start,below]{$j$} (1.7,2);
      \end{tikzpicture}};
  \end{tikzpicture}
\begin{tikzpicture}[baseline=8pt]
  \node at (0,0){\begin{tikzpicture} [scale=1.3] \node[scale=1.5] at
      (-.7,1){$\la$}; \draw[postaction={decorate,decoration={markings,
          mark=at position .5 with {\arrow[scale=1.3]{>}}}},very
      thick] (0,0) to[out=90,in=-90] node[at start,below]{$i$} (1,1)
      to[out=90,in=-90] (0,2) ;
      \draw[postaction={decorate,decoration={markings, mark=at
          position .5 with {\arrow[scale=1.3]{<}}}},very thick] (1,0)
      to[out=90,in=-90] node[at start,below]{$j$} (0,1)
      to[out=90,in=-90] (1,2);
    \end{tikzpicture}};

  \node at (1.7,0) {$=$}; \node at (3.2,0){\begin{tikzpicture}
      [scale=1.3]

      \node[scale=1.5] at (2.4,1){$\la$};

      \draw[postaction={decorate,decoration={markings, mark=at
          position .5 with {\arrow[scale=1.3]{>}}}},very thick] (1,0)
      to[out=90,in=-90] node[at start,below]{$i$ }(1,2) ;
      \draw[postaction={decorate,decoration={markings, mark=at
          position .5 with {\arrow[scale=1.3]{<}}}},very thick]
      (1.7,0) to[out=90,in=-90] node[at start,below]{$j$} (1.7,2);
    \end{tikzpicture}};
\end{tikzpicture}
\end{equation*}

\begin{equation*}\subeqn\label{through-ghosts1}
\begin{tikzpicture}[baseline=8pt]
  \node at (0,0){\begin{tikzpicture} [scale=1.3] \node[scale=1.5] at
      (-.7,1){$\la$}; \draw[postaction={decorate,decoration={markings,
          mark=at position .5 with {\arrow[scale=1.3]{>}}}},very
      thick] (0,0) to[out=90,in=-90] node[at start,below]{$i$} (1,1)
      to[out=90,in=-90] (0,2) ; \draw[dashed,
      postaction={decorate,decoration={markings, mark=at position .5
          with {\arrow[scale=1.3]{<}}}},very thick] (1,0)
      to[out=90,in=-90] node[at start,below]{$e$} (0,1)
      to[out=90,in=-90] (1,2);
    \end{tikzpicture}};

  \node at (1.7,0) {$=$}; \node[scale=1.1] at (2.3,0) {$u_{i;e}$};
  \node at (3.8,0){\begin{tikzpicture} [scale=1.3]

      \node[scale=1.5] at (2.4,1){$\la$};

      \draw[postaction={decorate,decoration={markings, mark=at
          position .5 with {\arrow[scale=1.3]{>}}}},very thick] (1,0)
      to[out=90,in=-90] node[at start,below]{$i$ }(1,2) ;
      \draw[dashed, postaction={decorate,decoration={markings, mark=at
          position .5 with {\arrow[scale=1.3]{<}}}},very thick]
      (1.7,0) to[out=90,in=-90] node[at start,below]{$e$} (1.7,2);
    \end{tikzpicture}};
\end{tikzpicture}
\begin{tikzpicture}[baseline=8pt]
  \node at (0,0){\begin{tikzpicture} [scale=1.3] \node[scale=1.5] at
      (-.7,1){$\la$};
      \draw[dashed,postaction={decorate,decoration={markings, mark=at
          position .5 with {\arrow[scale=1.3]{<}}}},very thick] (0,0)
      to[out=90,in=-90] node[at start,below]{$e$} (1,1)
      to[out=90,in=-90] (0,2) ;
      \draw[postaction={decorate,decoration={markings, mark=at
          position .5 with {\arrow[scale=1.3]{>}}}},very thick] (1,0)
      to[out=90,in=-90] node[at start,below]{$j$} (0,1)
      to[out=90,in=-90] (1,2);
    \end{tikzpicture}};

  \node at (1.7,0) {$=$}; \node[scale=1.1] at (2.3,0) {$u_{j;e}$};
  \node at (3.8,0){\begin{tikzpicture} [scale=1.3]

      \node[scale=1.5] at (2.4,1){$\la$};

      \draw[dashed,postaction={decorate,decoration={markings, mark=at
          position .5 with {\arrow[scale=1.3]{<}}}},very thick] (1,0)
      to[out=90,in=-90] node[at start,below]{$e$} (1,2);
      \draw[postaction={decorate,decoration={markings, mark=at
          position .5 with {\arrow[scale=1.3]{>}}}},very thick]
      (1.7,0) to[out=90,in=-90] node[at start,below]{$j$} (1.7,2);
    \end{tikzpicture}};
\end{tikzpicture}
\end{equation*}
\begin{equation*}\subeqn\label{through-ghosts2}
\begin{tikzpicture}[baseline=8pt]
  \node at (0,0){\begin{tikzpicture} [scale=1.3] \node[scale=1.5] at
      (-.7,1){$\la$}; \draw[postaction={decorate,decoration={markings,
          mark=at position .5 with {\arrow[scale=1.3]{<}}}},very
      thick] (0,0) to[out=90,in=-90] node[at start,below]{$i$} (1,1)
      to[out=90,in=-90] (0,2) ; \draw[dashed,
      postaction={decorate,decoration={markings, mark=at position .5
          with {\arrow[scale=1.3]{>}}}},very thick] (1,0)
      to[out=90,in=-90] node[at start,below]{$e$} (0,1)
      to[out=90,in=-90] (1,2);
    \end{tikzpicture}};

  \node at (1.7,0) {$=$}; \node[scale=1.1] at (2.3,0) {$t_{i;e}$};
  \node at (3.8,0){\begin{tikzpicture} [scale=1.3]

      \node[scale=1.5] at (2.4,1){$\la$};

      \draw[postaction={decorate,decoration={markings, mark=at
          position .5 with {\arrow[scale=1.3]{<}}}},very thick] (1,0)
      to[out=90,in=-90] node[at start,below]{$i$ }(1,2) ;
      \draw[dashed, postaction={decorate,decoration={markings, mark=at
          position .5 with {\arrow[scale=1.3]{>}}}},very thick]
      (1.7,0) to[out=90,in=-90] node[at start,below]{$e$} (1.7,2);
    \end{tikzpicture}};
\end{tikzpicture}
\begin{tikzpicture}[baseline=8pt]
  \node at (0,0){\begin{tikzpicture} [scale=1.3] \node[scale=1.5] at
      (-.7,1){$\la$};
      \draw[dashed,postaction={decorate,decoration={markings, mark=at
          position .5 with {\arrow[scale=1.3]{>}}}},very thick] (0,0)
      to[out=90,in=-90] node[at start,below]{$e$} (1,1)
      to[out=90,in=-90] (0,2) ;
      \draw[postaction={decorate,decoration={markings, mark=at
          position .5 with {\arrow[scale=1.3]{<}}}},very thick] (1,0)
      to[out=90,in=-90] node[at start,below]{$j$} (0,1)
      to[out=90,in=-90] (1,2);
    \end{tikzpicture}};

  \node at (1.7,0) {$=$}; \node[scale=1.1] at (2.3,0) {$t_{j;e}$};
  \node at (3.8,0){\begin{tikzpicture} [scale=1.3]

      \node[scale=1.5] at (2.4,1){$\la$};

      \draw[dashed,postaction={decorate,decoration={markings, mark=at
          position .5 with {\arrow[scale=1.3]{>}}}},very thick] (1,0)
      to[out=90,in=-90] node[at start,below]{$e$ }(1,2) ;
      \draw[postaction={decorate,decoration={markings, mark=at
          position .5 with {\arrow[scale=1.3]{<}}}},very thick]
      (1.7,0) to[out=90,in=-90] node[at start,below]{$j$} (1.7,2);
    \end{tikzpicture}};
\end{tikzpicture}
\end{equation*}
\item the KLR relations
  \cite[(2.6a-g)]{Webmerged}  replaced with the weighted KLR relations
  of Definition \ref{def-wKLR}. 
  \end{itemize}
  In both cases, we ignore the dotted ghosts; these are only necessary
  to label the plane so that
  $\mathfrak{sl}_2$ relations function correctly.
\end{itemize}

Note that if the loadings have each pair of points at least $s$ units apart, both these changes in
relations become irrelevant, and we recover the relations of the
original category $\tU$.

Note that $\mathcal{Y}_\vartheta$ has
a pair of commuting left and right actions of $\tU$, given by placing
diagrams in $\tU$ (drawn on loadings with points more than $s$ units
apart) to
the far left or far right of a diagram in $\mathcal{Y}_\vartheta$.

The morphism spaces in $\mathcal{Y}_\vartheta$ have a natural spanning
set analogous to that for $\tU$
described by Khovanov and Lauda, which we'll denote $Z_\vartheta$.  Each vector in $Z_\vartheta$ is indexed by
matching of the points of the two loadings such that points in the
different loadings have the same sign or in the same loading have
different signs.  The diagram is gotten by choosing a way of stringing
together the matched points, placing an arbitrary number of dots at a
fixed point on each strand, and then multiplying at the right by a
monomial in the bubbles (which are far enough apart to avoid any
interaction with ghosts).
\begin{lemma}
  The set $Z_\vartheta$ is a basis.
\end{lemma}

\begin{proof}
The proof that these relations span is very similar to that of Theorem
\ref{th:basis}: one can use the relations of Definition \ref{def-wKLR}
to remove any bigons, and show any two choices of the vectors in
$Z_\vartheta$ are the same, modulo diagrams with fewer crossings.

  Assume we have a non-trivial linear combination of diagrams in
  $Z_\vartheta$.  This must be gotten as a sum of the relations in the
  category as described earlier.  Now, attach the morphism that pulls
  all strands to the far right and separates them at least $s$ units
  from each other from each other to the top and bottom of the
  diagram.  The result of is a linear combination of morphisms in
  $\tU$.  Since every relation in $\mathcal{Y}_\vartheta$ remains a
  relation when a red line is dragged through it, or its ends are
  pulled further apart, the relations that we used to write this
  linear combination remain relations in $\tU$.  That is, the sum of
  diagrams we arrive at in $\tU$ is 0 as well.  However, we know by
  \cite [Thm. 4.10]{Webunfurl} that the analogous spanning set to
  $Z_\vartheta$ in $\tU$ is a basis, so when written in terms of these
  elements, it must be a trivial linear combination.

  Consider a diagram of $Z_\vartheta$ with a maximal number of
  crossings among those that appear in the linear combination. The
  diagram corresponding to the same matching (with some new dots)
  appears in our new linear combination, and no other diagram from the
  proposed basis could cancel it out.  Thus, it must have trivial
  coefficient in the original linear combination, contradicting the
  assumption that it did not.

  Thus, the set $Z_\vartheta$ is a basis; in particular, if we
  consider usual loadings as signed loadings with all signs negative,
  we get an injection of the weighted KLR algebra into the morphism
  space in $\mathcal{Y}$.
\end{proof}

Now, we apply a similar principle to have we have use many times in
\cite{Webmerged}; we call a signed loading {\bf unsteady}  like in the
unsigned case if it is horizontal composition of a purely black
loading with one containing all the red strands. We let $D\bar \walg^\vartheta(c)$ be the quotient of the
algebra spanned by double weighted KLR diagrams with $\EuScript{L}=0$
by the relations of the category $\mathcal{Y}_\vartheta$ and the ideal
generated by all unsteady signed loadings.

\begin{lemma}\label{lem:Morita}
  The natural map of algebras $\bar \walg^\vartheta(c) \to D\bar
  \walg^\vartheta(c)$ is a Morita equivalence. 
\end{lemma}

\begin{proof}
First, we must show that the morphism space in the quotient  $D\bar
  \walg^\vartheta(c)$ between two usual
loadings is the reduced steadied quotient of the weighted KLR
algebra.  This follows from a similar argument to
\cite[3.12]{Webmerged}.  As in \cite{Webmerged}, we call a signed
loading {\bf downward} if all its points have negative sign.   Consider any diagram
with downward top and bottom, and an unsteady loading at $y=\nicefrac
12$.  As in that proof, we can isotope the strands coming from the
unsteadying part of the loading so that they meet the line $y=\nicefrac
12$ again before meeting any part of the rest of the loading.  Now
isotope the diagram again, so that all but one of the resulting cups
is pushed below $y=\nicefrac
12$.  Now we see that our diagram is unsteadied by a loading beginning
with a $\pm i$ and then a $\mp i$.  Now, we can run the argument of
\cite[3.12]{Webmerged} to finish the proof.  This shows that the map
is injective.  

Now, in order to prove Morita equivalence, we need only prove that the
idempotent for any signed loading  $\Bi$ factors through downward loadings in
this quotient.  This is closely modeled on \cite[3.13]{Webmerged}.
We induct on the number of positive signs in $\Bi$, as well as the
length of the minimal permutation sending all positive signs to the
left and negative to the right.  If this permutation is the identity,
then the left-most point carries a positive sign, and without changing
the isomorphism type, we can pull it to the far left, so this loading
is trivial in $\EuScript{Y}_\vartheta$.  Thus, we must have a pair of
consecutive points where the leftward one carries a $-$ and the
rightward one carries a $+$.  We can move the rightward one to the
left through any ghosts or strands with
different labels using the relations
(\ref{through-strands}-\ref{through-ghosts2}).  If they carry the same label, then by the relation
\cite[(2.4c)]{Webmerged}, $e_{\Bi}$ factors through loadings 
where these points have switched (lowering the length of the
permutation) plus some number where they have been removed (lowering
the number of $+$'s).  By induction, this map is a Morita equivalence.
\end{proof}

\begin{proof}[Proof of Theorem \ref{th:categorical-action}]
The set of morphisms that factor through unsteady loadings is closed
under horizontal composition on the right with 1-morphisms in $\tU$; adding anything on the
right side of a diagram will not change the unsteady property.  Thus
  the category of projective modules over $D\bar{\walg}^\vartheta(c)$
  is a quotient of $\mathcal{Y}_\vartheta$ by a set of morphisms which
  are closed under horizontal composition on the right with 1-morphisms in $\tU$.  That is,
  $D\bar{\walg}^\vartheta(c)\operatorname{-pmod}$ carries a natural action of
 $\tU$, induced by horizontal composition. 

 By the
  Morita equivalence of Lemma \ref{lem:Morita}, the same is true of
  $\bar{\walg}^\vartheta(c)\operatorname{-pmod}$.  This action is induced by
  bimodules $\beta_u$ for $u\colon \mu\to \nu$ spanned by diagrams
  like those 
  drawn schematically as below:
\begin{equation}\label{DRu-schematic}
  \tikz[very thick,baseline]{\draw[postaction={decorate,decoration={markings,
    mark=at position .5 with {\arrow[scale=1.3]{<}}}}] (0,-1) to[out=90,in=-90] (-.5,1);
\draw[postaction={decorate,decoration={markings,
    mark=at position .5 with {\arrow[scale=1.3]{<}}}}] (.25,-1) to[out=90,in=-90] (0,1);
\draw[postaction={decorate,decoration={markings,
    mark=at position .5 with {\arrow[scale=1.3]{>}}}}]  (1,1) to[out=-90,in=-90] (2.5,1);
\draw[postaction={decorate,decoration={markings,
    mark=at position .5 with {\arrow[scale=1.3]{<}}}}]  (2.25,-1) to[out=90,in=-90] (3.5,1);
\draw[postaction={decorate,decoration={markings,
    mark=at position .5 with {\arrow[scale=1.3]{<}}}}]  (2.5,-1) to[out=90,in=-90] (-.25,1);
\draw[postaction={decorate,decoration={markings,
    mark=at position .5 with {\arrow[scale=1.3]{<}}}}]  (2.75,-1) to[out=90,in=-90] (1.25,1);
\node at (1.5, -.7){$\cdots$};
\node at (.5, .7){$\cdots$};
\node at (3, .7){$\cdots$};
\draw[decorate,decoration=brace,-] (3,-1.25) --
    node[below,midway]{$\bar{\walg}^\vartheta _\mu (c)$-action} (-.25,-1.25);
\draw[decorate,decoration=brace,-] (-.75,1.25) --
    node[above,midway]{$\bar{\walg}^\vartheta _\nu (c)$-action} (1.5,1.25);
\draw[decorate,decoration=brace,-] (2.25,1.25) --
    node[above,midway]{$u\colon \mu\to \nu$} (3.75,1.25);
  }
\end{equation} 
with all the relations of $\mathcal{Y}_\vartheta$ and of
$\bar{\walg}^\vartheta(c)$ imposed.
\end{proof}

Exactly as in \cite[Prop. 6.7]{Webmerged}, we have that:

\begin{proposition}
  The functor of tensor product with $B^{\vartheta,\vartheta'}(c)$
  commutes naturally with the action of $\tU$.
\end{proposition}

\excise{
We'll show later that $\dalg^\vartheta$ is
free over $S$.  In particular, if we consider the specialization
$\dalg^\vartheta\otimes_{S}\K$ at the unique graded maximal ideal, we
simply recover the algebra $\alg^\vartheta$ for the usual choice of
$Q_{ij}$.

\excise{\begin{example}
  If we choose a trivial weighting\footnote{I know I said we would
    only consider generic weightings.  See the earlier discussion of hobgobblins.}, then the steadied quotient of
  $\alg^\la$ will just be the usual cyclotomic KLR algebra for the
  weight $\la$.  The algebra $\dalg^\la$ is the deformation of
  this obtained by changing the usual cyclotomic relation
  $e(\Bi)y_1^{\la^{i_1}}=0$ to a polynomial relation
  \[e(\Bi)\prod_{\la_i=\om_{i_1}} (y_1-z_i)=0.\]
This is a Galois extension of the universal cyclotomic quotient
discussed in \cite[\S 1]{WebCTP}, via the $\prod_iS_{\la^i}$ action
permuting the $z_i$'s that correspond to the same fundamental weight.
\end{example}

\begin{proposition}\label{prop:cyclo-flat}
  The algebra $\dalg^\la$ is free over $S$.
\end{proposition}
\begin{proof}
  It suffices to prove that $\dalg^\la$ is flat over $S$, and thus
  that its residue at every closed point has the same dimension as at
  the generic point.  But at every closed point, we obtain a
  categorical representation of $\mathfrak{sl}_e$ (with different
  polynomials $Q_*$) and thus, the Euler form on representations is
  always the Shapovalov form by \cite[1.10]{WebCTP}, and the dimension
  of the algebra can be written as a sum of Shapovalov products
  independent of the point.  Thus, we are done. 
\end{proof}}

    For a field $K$, 
consider a $K$-point $\chi\colon \K[z_1,\dots,
    z_\ell,h]\to K$.  Let $\dalg^\vartheta(\chi)$ denote the base
    change of $\dalg^\vartheta$ at this point.  
    \begin{definition}
      We let \[\Gamma_\chi = \{(i, \chi(z_j)+m\chi(h))\in \Z/e\Z\times K\mid
      m+r_j\equiv i\pmod e\}.\] This set can be given a graph
      structure where $(i,g)$ is adjacent to $(i+1,g-\chi(h))$.  This
      graph has an obvious induced weighting $\vartheta_\chi $.   Note
      that this graph structure is precisely that given in \cite[\S \ref{w-sec:canon-deform}]{WebwKLR}.

Let
      $\fg_\chi $ be the Kac-Moody algebra with this Dynkin diagram.  Let
      $\tU_\chi $ be the 2-category attached to this diagram for the
      polynomials $Q_e(u,v)=u-v$.  Let $\bla^\chi$ be the ordered $\ell$-tuple of highest
      weights given by the fundamental weight for $(r_j,\chi(z_j))$ for $j=1,\dots,\ell$.  
    \end{definition}
For example, if $K$ is the quotient field of
$\K[z_1,\dots,z_\ell,h]$ and $a$ the tautological point, then $\Gamma_\chi $ is a union of $\ell$
different infinite strings (i.e. $\fg_\chi \cong
(\mathfrak{sl}_\infty)^\ell$) and each element of $(r_j,a(z_i))$ is on
a different string.
This graph arises naturally since:
\begin{lemma}\label{lem:y-eigenvalue}
  If $a$ is an
  eigenvalue of $y_j$ on the image of $e_\Bi$ such that $i_j=i$, then
  $(a,i)\in \Gamma_\chi$.
\end{lemma}
\begin{proof}
  Fix a loading $\Bi$ and let $\Bi'$ be the loading where we have
  moved the $j$th point to the left of the closest ghost or black
  strand on its left, {\it or} its ghost moves to left of a black
  strand, whichever moves it less.  This makes it the $j'$th strand
  for $j'=j$ or $j-1$. Now consider the diagram which gives $\Bi$ at
  the top and bottom and $\Bi'$ at the horizontal slice $y=\nicefrac
  12$, with a minimal number of crossings.  This simply forms a single
  bigon, either between the $j$th strand, and the ghost or strand
  immediately to its left, or between the ghost of the $j$th strand
  and the strand to its left.  Note that for these purposes, we are
  considering the red lines to be ghosts.  If $F_{\Bi',j'}(x)$ is the
  minimal polynomial of $y_{j'}$ acting on the image of $e_{\Bi'}$,
  inserting $F_{\Bi',j'}(y_{j'})$ at $y=\nicefrac 12$ gives 0.  On the
  other hand, if we use the relations
  (\ref{strand-bigon}--\ref{ghost-bigon1}) or (\ref{d-ghost bigon}) to
  cancel the bigon, then we will either obtain
  \begin{enumerate}
  \item $F_{\Bi',j'}(y_{j})=0$ if the crossings where degree 0
  \item $(y_{j}-y_k+\chi(h)) F_{\Bi',j'}(y_{j})=0$ if the crossings
    where of the ghost of the $j$th strand over the $k$th and
    $i_k+1=i_j$.
  \item $-(y_j-y_k-\chi(h)) F_{\Bi',j'}(y_{j}) =0$ if the crossings
    where of the ghost of the $k$th strand over the $j$th and
    $i_k-1=i_j$.
  \item $(y_j-z_k) F_{\Bi',j'}(y_{j}) =0$ if the crossing was over the
    $k$th red strand, and $\la_k=\omega_{i_j}$.
  \end{enumerate}
  Obviously, in the case (1), the spectrum of $y_j$ on the image of
  $e_{\Bi}$ is a subset of that for $y_{j'}$ on the image of
  $e_{\Bi'}$; similarly, in case (4), the spectrum of $y_j$ on the
  image of $e_{\Bi}$ satisfies the same constraint, but may also
  contain $\chi(z_k)$.

  Now consider case (2), and let $a$ be an eigenvalue of $y_j$ on the
  image of $e_{\Bi}$. In this case, $(a-y_k+\chi(h))
  F_{\Bi',j'}(a)=0$. It may be that $F_{\Bi',j'}(a)=0$, and $a$ is in
  the spectrum of $y_{j'}$ on the image of $e_{\Bi'}$.  If not, then
  $(a-y_k+\chi(h))$ is not invertible, so $a+\chi(h)$ is in the
  spectrum of $y_k$ on $e_\Bi$ or $e_{\Bi'}$.  Similarly, in case (3)
  we obtain almost the same result, except with $a-\chi(h)$.

  Thus, if $a$ is an eigenvalue of $y_j$ on the image of $e_\Bi$ such
  that $i_j=i$, then $a=\chi(z_j)$ or there is a loading such that the
  sum of the positions of the black strands is smaller such that either:
  \begin{itemize}
  \item $a$ is an eigenvalue of $y_m$ on the image of $e_\Bi$ such
    that $i_m=i$ or
  \item $a\pm \chi(h)$ is an eigenvalue of $y_m$ on the image of
    $e_\Bi$ such that $i_m=i\pm 1$.
  \end{itemize}
  If the sum of the position of the black strands is small enough,
  some block 
  of them must be far enough left to unsteady the loading, so the image
  in the steadied quotient is trivial. This shows that if $a$ is an
  eigenvalue of $y_j$ on the image of $e_\Bi$ such that $i_j=i$, then
  $(a,i)\in \Gamma_\chi$.
\end{proof}

   Let $\eF_{(i,g)}M\subset \eF_iM$ and $\eE_{(i,g)}M\subset
   \eE_iM$ be the stable kernels of  the endomorphism $y-g$.  
These functors are obviously biadjoint.  

We can
   define a nilpotent natural transformation $y_\chi =(y-g)\colon
   \eF_{(i,g)}\to \eF_{(i,g)}$.  We can also define a natural
   transformation $\psi_\chi \colon \eF_{(i,g)}\eF_{(j,g')}\to
   \eF_{(j,g')}\eF_{(i,g)}$ as follows: 

   \begin{itemize}
   \item  If $i=j$ and $g= g'$ then we let $\psi_\chi $ be the restriction to
   $\eF_{(i,g)}\eF_{(j,g')}$ of $\psi$. 
\item   If $i=j$ and $g\neq g'$, then
   we let $\psi_\chi $ be the restriction to $\eF_{(i,g)}\eF_{(j,g')}$ of
   $\sigma=y_1\psi-\psi y_1$; note that $\sigma^2=1$.
\item If $i\neq j$, then consider $P_{(i,g),(j,g')}(y_1,y_2)$,  the product  over edges from
$i$ to $j$ that don't lift to an edge from $(i,g)$ to $(j,g')$ of
$\prod_e \dQ_e(y_1,y_2)$.  This is a polynomial that always acts
invertibly on $\eF_{(i,g)}\eF_{(j,g')}$  since $\dQ_e(g,g')\neq 0$ if
$e$ doesn't lift to $\Gamma_\chi $.   Furthermore, it satisfies \[P_{(i,g),(j,g')}(y_1,y_2) P_{(j,g'),(i,g)}(y_2,y_1) =\dQ_{i,j}(y_1,y_2)
/Q_{(i,g),(j,g')}(y_1-g,y-g').\] 
   \end{itemize}

We consider the natural transformation $\psi_\chi \colon \eF_{(i,g)}\eF_{(j,g')}\to
     \eF_{(i+1,g')}\eF_{(i,g)}$ given by the restriction of $P_{(i,g),(j,g')}^{-1}\psi$ to this
     summand.
  \begin{theorem}
The category $\dalg^\vartheta(\chi)\umod$ carries a categorical
action of  $\fg_\chi $ where the functor $\eF_{(i,g)}$ is the stable
kernel of
$y-g$  acting on $\eF_i$.   
  \end{theorem}
  \begin{proof}
Of course, all relations from $\tU_{\mathfrak{sl}_2}$ follow
immediately.  Thus, the main thing that we need to check is that the elements $y_\chi $ and
$\psi_\chi $ satisfy the KLR relations. This is easily confirmed by matching
these with the action on polynomial rings given in
\cite[3.12]{Rou2KM}, which is done in  \cite[\ref{w-completion}]{WebwKLR}.   Thus, the existence of the desired action
follows immediately from \cite[1.1]{CaLa}.
  \end{proof}

We'll be interested in the algebra $\alg^{\vartheta_\chi }(K)$, the 
steadied quotient of the weight KLR algebra with coefficients in the
field $K$ for the graph $\Gamma_\chi$ with the weighting pulled back
by the map $\Gamma_\chi\to \Gamma$.
We'll actually be interested in a different grading on this algebra
than the usual one.   In the grading we have described, crossings between strands and
ghosts having degree given by 1 if the label on the strand is the tail
of the edge, and 0 otherwise. 

In our new grading, for each edge in $\Gamma$ from $i$ to
$j$, we add 1 to the grading of a crossing where a strand colored $i$ passes left
through a $j$-colored strand, and subtract one from the opposite
crossing.

We note that this grading change does not change the graded category
of modules; there is a Morita equivalence given by a grading shifted
version of the diagonal bimodule. 

  \begin{proposition}
We have an isomorphism of graded algebras $\dalg^\vartheta(\chi)\cong \alg^{\vartheta_\chi}(K)$.
  \end{proposition}
  \begin{proof}
  This follows from the isomorphism between completed weighted KLR
  algebras in \cite[\ref{w-completion}]{WebwKLR}.  As noted there,
  this map also induces an isomorphism between completed steadied quotients, and
  obviously extends to the reduced case.  Thus, we need only check
  that the steadied quotients are not changed by completion.  For
  $\alg^{\vartheta_\chi}(K)$, this is clear from the fact that all
  $y_i$'s will act nilpotently on the reduced steadied quotient if all
  the polynomials $Q_{ij}$ have trivial constant term.  On the other
  hand, in  $\dalg^\vartheta(\chi)$, the $y_i$'s are not nilpotent,
  but only finitely many generalized eigenvalues for them can occur.
Lemma \ref{lem:y-eigenvalue} shows that these all lie in
$\Gamma_\chi$, so completion at these points leaves the steadied
quotient unchanged.
  \end{proof}

\begin{corollary}\label{cor:ell-strips}
  If $\Gamma_\chi $ consists of $\ell$ distinct copies of $A_\infty$, then
  $\dalg^{\vartheta}(\chi)$ is semi-simple.  In particular, this
  happens if $K$ is the quotient field of $\K[z_1,\dots,z_\ell,h]$
  and $\chi$ the tautological point.
\end{corollary}

}

\subsection{Relations to tensor product algebras}
\label{sec:relat-tens-prod}

Fix a list of highest weights $\bla=(\la_1,\dots, \la_\ell)$.
Choose any sequence of real numbers $\varpi_1 < \cdots < \varpi_\ell$,
and consider the weighting on $\Gamma_\la$ where all old edges have degree 0,
and there are $\al_i^\vee(\la_j)$ new edges with weight $\varpi_j$ connecting
$0$ to $i$.   We denote these edges $e_{i,j,1},\dots, e_{i,j,\al_i^\vee(\la_j)}$ Recall that in \cite[\S 4]{Webmerged}, the author defined algebras
$T^\bla$ and $\tilde T^\bla$  attached to the list $\bla$.

\begin{theorem}
  The algebra $\tilde T^\bla_{\la-\nu}$ is the reduced quotient $\bar \walg^\varpi_{\tilde \nu}$ of $\walg^\varpi_{\tilde \nu}$.
  The map replaces
  the ghosts of the $0$-labelled strand with red strands, decorated by
  the weights $\la_1$ through $\la_\ell$ if the $0$-labelled strand has
  no dots on it, and sends the diagram to 0 if there are any dots on
  the $0$-labelled strand.
\end{theorem}
\begin{proof}
  All relations between black strands satisfy the KLR relations in
  both cases.  When we undo a bigon between the $i$-labelled $k$th
  strand and the $p$th $0$-labelled ghost (from the left) where the
  $m$th strand is $0$-labelled, we multiply by
  $(y_k-y_m)^{\al_i^\vee(\la_p)}$, which becomes $y_k
  ^{\al_i^\vee(\la_p)}$ after setting $y_m=0$.  Similarly, if a ghost
  passes through a crossing of the $k$th and $k+1$st strands, the
  correction term is the opened crossing
  times \[\partial_{k,k+1}((y_k-y_m)^{\al_i^\vee(\la_p)})=\frac{(y_k-y_m)^{\al_i^\vee(\la_p)}-(y_{k+1}-y_m)^{\al_i^\vee(\la_p)}}{y_{k}-y_{k+1}},\]
  which becomes
  $y_k^{\al_i^\vee(\la_p)}+y_k^{\al_i^\vee(\la_p)-1}y_{k+1}+\cdots +
  y_{k+1}^{\al_i^\vee(\la_p)}$ after setting $y_m=0$, which is exactly
  the relation expected from \cite[(4.1a)]{Webmerged}.  Finally, in all other
  triple points, there is no correction term in either set of
  relations.  This confirms all the relations of $\tilde T^\bla$.

  Thus, turning all ghosts into red strands gives a surjective map 
  $\bar \walg^\vartheta_{\tilde \nu} \to \tilde T^\bla_{\la-\nu}$.  Note that this map sends basis vectors to basis
  vectors for the diagram bases of these algebras, and thus is an
  isomorphism.
\end{proof}
\begin{theorem}
  The tensor product algebra $T^\bla_{\la-\nu}$ is the reduced steadied quotient
  of the weighted algebra $\walg^\varpi_{\tilde{\nu}}(c)$ for $\Gamma_w$.
  Similarly, the bimodule $B^{\vartheta,\vartheta'}(c)$ for two different
  tensor product weightings is exactly $\mathfrak{B}^{\si}$, where $\si$
  is the positive braid lift of the permutation sending the total
  order on new edges by weight in  $\vartheta$ to that induced by
  weight in $\vartheta'$. 
\end{theorem}
\begin{proof}
  Note that if $\nu'+\nu''=\tilde \nu$, then $\nu'>_c\nu''$ if
  and only if the $0$-component of $\nu'$ is 0 and that of $\nu''$ is
  1.  Thus, the unsteady ideal is generated by diagrams where a block
  of strands all labeled with old vertices are ``much further'' left
  than the 0-labelled strands.  This obviously corresponds to the
  violating ideal as defined in \cite[\S 4]{Webmerged}, so we have the
  desired isomorphism.
\end{proof}

In this case, we can apply the canonical deformation discussed in
Section \ref{sec:canon-deform}, which gives algebras like those
appearing in \cite[\S \ref{m-sec:univ-quant}]{Webmerged}.  Let us take
this deformation for the weighted KLR algebra of Crawley-Boevey
quiver, and set all coefficients $z_{e,a,b}=0$ for $e$ an old edge
(one from the original quiver).  We're left with the parameter $z_{e,0,0}$
for each new edge;  we'll abbreviate $z_{i,j,k} =-z_{e_{i,j,k},0,0}.$  This results in a deformation of the algebra
$T^\bla$, where the number of parameters $\{z_{i,j,k}\}$ is the number of new edges,
that is, $\rho^\vee(\la)$.  

We can easily describe how the relations of $T^\bla$ deform in this
case.  For each $i\in \Gamma$ and $j\in [1,\ell]$, let
$p_{i,j}(u)=(u-z_{i,j,1})\cdots (u-z_{i,j,\al_i^\vee(\la_j)})$.  The
relations
\cite[(\ref{m-red-triple-correction},\ref{m-cost})]{Webmerged} thus
deform to:\newseq
  \begin{equation*}\label{red-triple}\subeqn
    \begin{tikzpicture}[very thick,baseline=-2pt]
      \draw [wei] (0,-1) -- node[below,at start]{$\la_j$} (0,1); \draw(.5,-1)
      to[out=90,in=-30] node[below,at start]{$i$} (-.5,1); \draw(-.5,-1)
      to[out=30,in=-90] node[below,at start]{$i$} (.5,1);
    \end{tikzpicture}- \begin{tikzpicture}[very
      thick,baseline=-2pt] \draw [wei] (0,-1) --node[below,at start]{$\la_j$} (0,1);
      \draw(.5,-1) to[out=150,in=-90] node[below,at start]{$i$} (-.5,1);
      \draw(-.5,-1) to[out=90,in=-150] node[below,at start]{$i$} (.5,1);
    \end{tikzpicture}
    =\sum_{p=1}^{\la^i}\sum_{a+b=p-1}e_{\ell-p}(-z_{i,j,*}) \cdot \Bigg(\begin{tikzpicture}[very thick,baseline=-2pt]
      \draw [wei]  (0,-1) -- (0,1);
      \draw(.5,-1) to[out=90,in=-90] node[midway,circle,
      fill=black,inner sep=2pt,label=right:{$b$}]{} (.5,1);
      \draw(-.5,-1) to[out=90,in=-90] node[midway,circle,
      fill=black,inner sep=2pt,label=left:{$a$}]{} (-.5,1);
    \end{tikzpicture}\Bigg).
  \end{equation*}
  The RHS can alternately by written as $(p_{i,j}(y_{r+1})-p_{i,j}(y_r))/({y_{r+1}-y_r})$.
  \begin{equation*}\label{cost}\subeqn
  \begin{tikzpicture}[very thick,baseline=1.6cm]
    \draw (-2.8,0)  +(0,-1) .. controls (-1.2,0) ..  +(0,1) node[below,at start]{$i$};
       \draw[wei] (-1.2,0)  +(0,-1) .. controls (-2.8,0) ..  +(0,1) node[below,at start]{$\la_j$};
           \node at (-.3,0) {$=p_{i,j}$};  \node[scale=1.5] at (.5,0) {$\Bigg($}; \node[scale=1.5] at (3.5,0) {$\Bigg)$}; 
    \draw[wei] (2.8,0)  +(0,-1) -- +(0,1) node[below,at start]{$\la_j$};
       \draw (1.2,0)  +(0,-1) -- +(0,1) node[below,at start]{$i$};
       \fill (1.2,0) circle (3pt);
          \draw[wei] (-2.8,3)  +(0,-1) .. controls (-1.2,3) ..  +(0,1) node[below,at start]{$\la_j$};
  \draw (-1.2,3)  +(0,-1) .. controls (-2.8,3) ..  +(0,1) node[below,at start]{$i$};
           \node at (-.3,3) {$=p_{i,j}$};\node[scale=1.5] at (.5,3) {$\Bigg($}; \node[scale=1.5] at (3.5,3) {$\Bigg)$}; 
    \draw (2.8,3)  +(0,-1) -- +(0,1) node[below,at start]{$i$};
       \draw[wei] (1.2,3)  +(0,-1) -- +(0,1) node[below,at start]{$\la_j$};
       \fill (2.8,3) circle (3pt);
  \end{tikzpicture}
\end{equation*}

\subsection{Relation to quiver Schur algebras}
\label{sec:relat-quiv-schur}

When $\Gamma$ is a cycle with $n$ vertices, then we have some
particularly interesting behavior.  The choice of weightings (up to
equivalence) is 1-dimensional, since $H^1(\Gamma;\R)\cong \R$.
Weightings are distinguished by the sum of the weights over an
oriented cycle. 
We can identify $\Gamma=\Z/n\Z$, with an edge $i\to i+1$; we 
 let $\vartheta_e=k$, a constant.

  Choose $0<\epsilon\ll |k| \ll s$.  For each vector composition
  $\hat{\bmu}=\bmu^{(1)},\dots, \bmu^{(m)}$, we associate the
  following loading $\Bj(\hat{\bmu})$: take the
  residue sequence (as defined in \cite[(3)]{SWschur}) for this
  sequence, and for each entry of  the $j$th block of the residue sequence
  $p_1,\dots$, add a points at $js+\ell \epsilon$ labeled with
  $p_\ell$ (so, we assume that $\epsilon<|k|/\ell_{max}$).  Thus, for
  each piece of the vector composition, we have a cluster of points in
  the loading whose labels sum to that piece, and the clusters are
  very far apart.  Now take the idempotent mapping the loading to
  itself which on the like-labelled strands of each piece of the
  loading does the idempotent which acts on polynomials
  by projecting to symmetric polynomials.  Note that within each
  block, rearranging strands will result in isomorphic idempotents.  

  \begin{example}
    If $\hat{\bmu}= (1,1,2), (2,0,0)$ and $k>0$, the loading is 
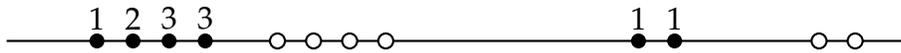
\[\tikz[thick]{\draw (-6,0) -- node [pos=.1, fill=black,circle,inner
     sep=2pt,outer sep=2pt]{} node [pos=.1,above]{$1$} node [pos=.14, fill=black,circle,inner
     sep=2pt,outer sep=2pt]{} node [pos=.14,above]{$2$} node [pos=.18, fill=black,circle,inner
     sep=2pt,outer sep=2pt]{} node [pos=.18,above]{$3$}
node [pos=.22, fill=black,circle,inner
     sep=2pt,outer sep=2pt]{} node [pos=.22,above]{$3$}
node [pos=.7, fill=black,circle,inner
     sep=2pt,outer sep=2pt]{} node [pos=.7,above]{$1$} 
node [pos=.74, fill=black,circle,inner
     sep=2pt,outer sep=2pt]{} node [pos=.74,above]{$1$} 
node [pos=.3, fill=white,draw=black,circle,inner
     sep=2pt,outer sep=2pt]{}
node [pos=.34, fill=white,draw=black,circle,inner
     sep=2pt,outer sep=2pt]{}
node [pos=.38, fill=white,draw=black,circle,inner
     sep=2pt,outer sep=2pt]{}
node [pos=.42, fill=white,draw=black,circle,inner
     sep=2pt,outer sep=2pt]{}
node [pos=.9, fill=white,draw=black,circle,inner
     sep=2pt,outer sep=2pt]{}
node [pos=.94, fill=white,draw=black,circle,inner
     sep=2pt,outer sep=2pt]{}
  (6,0); }\]
where we
    represent ghosts by hollow circles.
  \end{example}

  There are some obvious idempotents acting on each of these loadings $\Bj(\hat{\bmu})$;
  let $e_{\Bj(\hat{\bmu})}'$ be the idempotent that acts on
  $\Bj(\hat{\bmu})$ by applying the idempotent $e_n$ projecting to
  symmetric polynomials to the like-labelled points in each cluster.
  Let $e_{QS}$ be the sum of the idempotents $e_{\Bj(\hat{\bmu})}'$.

\begin{theorem}\label{th:schur-1}
  The algebra $e_{QS}\walg^\vartheta_\nu e_{QS}$ is isomorphic to the quiver Schur
  algebra $A_\bd$ defined in \cite{SWschur}.  
\end{theorem}
\begin{proof}
  This isomorphism sends
  the split of \cite{SWschur} to the analogous splitting of the
  idempotents we described  without crossing any like-labelled
    strands, and the merge to merging with crossing all pairs of
    like-labelled strands from the two merging pieces.   These are
    shown in Figure \ref{QS-compare}. It's easily
    checked that these act exactly as in  \cite[ 3.4]{SWschur}; in
    fact this is already shown in \cite[(23)]{SWschur}.  Thus,
    $A_\bd$ injects into this space, and the graded dimensions of the two
    algebras coincide, since the dimensions of the summands going between vector
    compositions $\hat{\bmu}$ and $\hat{\bmu}'$ both count double cosets for the
    subgroups of $S_m$ corresponding to the vector compositions.
\end{proof}

\begin{figure}
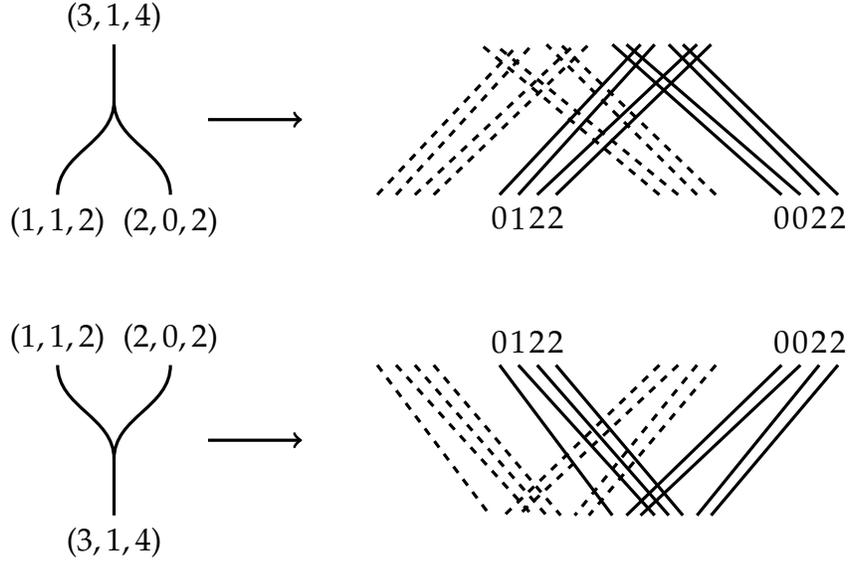

  \centering
  \[
\tikz[very thick,xscale=2.5,yscale=2.5]{
\draw (-1,0) node[below] {$(1,1,2)$} to [out=90,in=-90](-.7,.5)
(-.4,0) node[below] {$(2,0,2)$} to [out=90,in=-90] (-.7,.5)
(-.7,.5) -- (-.7,.8) node[above] {$(3,1,4)$};
\draw[->] (-.2,.4)--(.3,.4);  
\draw (1.35,0) -- node[below,at start]{$0$}(2.1,.8);
\draw (1.45,0) -- node[below,at start]{$1$} (2.175,.8);
\draw (1.55,0) -- node[below,at start]{$2$} (2.4,.8);
\draw (1.65,0) -- node[below,at start]{$2$} (2.475,.8);
\draw (2.85,0) -- node[below,at start]{$0$} (1.95,.8);
\draw (2.95,0) -- node[below,at start]{$0$} (2.025,.8);
\draw (3.05,0) -- node[below,at start]{$2$} (2.25,.8);
\draw (3.15,0) -- node[below,at start]{$2$} (2.325,.8);
\draw[dashed] (.7,0) -- (1.45,.8);
\draw[dashed] (.8,0) --  (1.525,.8);
\draw[dashed] (.9,0) -- (1.75,.8);
\draw[dashed] (1,0) -- (1.825,.8);
\draw[dashed] (2.2,0) -- (1.25,.8);
\draw[dashed] (2.3,0) -- (1.325,.8);
\draw[dashed] (2.4,0) -- (1.6,.8);
\draw[dashed] (2.5,0) -- (1.675,.8);
}\]

  \[
\tikz[very thick,xscale=2.5,yscale=2.5]{
\draw (-1,0) node[above] {$(1,1,2)$} to [out=-90,in=90](-.7,-.5)
(-.4,0) node[above] {$(2,0,2)$} to [out=-90,in=90] (-.7,-.5)
(-.7,-.5) -- (-.7,-.8) node[below] {$(3,1,4)$};
\draw[->] (-.2,-.4)--(.3,-.4);  
\draw (1.35,0) -- node[above,at start]{$0$}(1.95,-.8);
\draw (1.45,0) -- node[above,at start]{$1$} (2.175,-.8);
\draw (1.55,0) -- node[above,at start]{$2$} (2.25,-.8);
\draw (1.65,0) -- node[above,at start]{$2$} (2.325,-.8);
\draw (2.85,0) -- node[above,at start]{$0$} (2.025,-.8);
\draw (2.95,0) -- node[above,at start]{$0$} (2.1,-.8);
\draw (3.05,0) -- node[above,at start]{$2$} (2.4,-.8);
\draw (3.15,0) -- node[above,at start]{$2$} (2.475,-.8);
\draw[dashed] (.7,0) -- (1.3,-.8);
\draw[dashed] (.8,0) --  (1.525,-.8);
\draw[dashed] (.9,0) -- (1.6,-.8);
\draw[dashed] (1,0) -- (1.675,-.8);
\draw[dashed] (2.2,0) -- (1.375,-.8);
\draw[dashed] (2.3,0) -- (1.45,-.8);
\draw[dashed] (2.4,0) -- (1.75,-.8);
\draw[dashed] (2.5,0) -- (1.825,-.8);
}\]
  \caption{The comparison map with quiver Schur algebras}
\label{QS-compare}
\end{figure}
 
More generally, there are algebras, defined in \cite[\S 4]{SWschur}, which
mix together features of the quiver Schur algebras above with those of
the tensor product algebras.  These arise
from the Crawley-Boevey quiver $\Gamma_{\Bw}$ for the $n$-cycle and
some dimension vector $\Bw$.  As before, choose a weighting
$\vartheta$, and let $k$ be the sum of the
weights on the cycle. 

For each pair of new edges $e_1,e_2$, one can consider all the closed
paths which leave the CB vertex using $e_1$ and arrive using $e_2$.
If these connect to the same vertex in the cycle, there's a unique
such path which isn't self-intersecting (just the bigon), and
otherwise, there are two which go around the cycle in opposite
directions.  
We call a choice of $\vartheta$ {\bf
  well-separated} for a dimension vector $\bd$ if for any pair of
new edges, the absolute value of the weight assigned to any
non-self-intersection loop which starts with one and ends with the
other is greater than $k(\sum_{i\in I} d_i)$.  

In a
well-separated weighting, we can order the new edges according to
their weight unambiguously, since the weight of the two
non-intersecting paths have the same sign (otherwise, we might have
one positive, and one negative). 
We can consider the new edges in increasing order.  Each one connects
to a node in the cycle, to which we have associated a fundamental
weight. Thus, we
obtain a list of
fundamental weights $\bla=\{\la_1,\dots,\la_\ell\}$, where $\ell$ is the
total number of new edges, usually called the {\bf level} in this context.  Furthermore, to each list $\grave\mu=(\hat{\bmu}(0),
\hat{\bmu}(1),\dots,\hat{\bmu}(\ell))$ of vector compostions, we can
associate a loading as follows: we place a copy of the loading for
$\hat{\bmu}(i)$ and its idempotent $e'_{\hat{\bmu}(i) }$ (as constructed above) shifted by the position $b_i$
of the $i$th red strand.  That is, we place it on the real line just
right of the $i$th red strand.  

Let $e_{QS;\bla}$ be the sum of idempotents
attached to these loadings. 

\begin{theorem}\label{qSchur}
  If we choose $\vartheta$ well-separated, then the subalgebra $e_{QS;\bla}
  \bar \walg^\vartheta_\nu e_{QS;\bla}$ of the reduced quotient is
  the extended quiver Schur algebra $\tilde A^\bla_\bd$ associated  to
  $\bla$, and the subalgebra $e_{QS;\bla} \bar \walg^\vartheta_\nu(c) e_{QS;\bla}$ of
  reduced steadied quotient is isomorphic to $A^\bla_\bd$, and thus
  isomorphic to a cyclotomic $q$-Schur algebra. 
\end{theorem}
\begin{proof}
The first isomorphism is exactly as in Theorem
\ref{th:schur-1}; we simply note that the action of these operators on
the appropriate symmetric polynomials exactly match those of $\tilde
A^\bla_\bd$.

The steadied quotient exactly kills all idempotents where
$\hat{\mu}(0)\neq 0$, and thus coincides with the cyclotomic quotient.
\end{proof}

In fact, both these inclusions of subalgebras induce Morita
equivalences, but we omit a proof of this fact; the construction of a
cellular basis in \cite[\S 3]{WebRou} shows that no simple
representation is killed by this idempotent.
It is more natural to consider this in the context of a general weighting of an affine
quiver, which is probably the most interesting and powerful
application of the theory developed here; we develop this further
in \cite{WebRou}.

  \section{The geometry of quivers}
  \label{sec:geometry-quivers}

Throughout this section, we assume that $\Gamma$ is a
  multiplicity-free quiver; that is, we assume that $c_e=1$ and $Q_e(u,v)=u-v$ for all
  oriented edges, though we do allow multiple edges between the same pair of
  vertices.  Furthermore, for simplicity, we'll assume   throughout this section that $\operatorname{char}(\K)=0$.

  \subsection{Loaded flag spaces}
  \label{sec:loaded-flag-spaces}

If $\nu=\sum d_i\al_i$, we let $V_i=\C^{d_i}$, $V=\oplus_i
  V_i$ and let \[E_\nu= \bigoplus_{e\in
    \Omega}\Hom(V_{t(e)},V_{h(e)}).\] This vector space has a natural
  action of $G_\nu=\prod_{i\in I}GL(V_i)$ by pre- and
  post-composition.

The vector $\bd=(d_i)_{i\in \Gamma}$
is called the {\bf dimension vector}, and we will freely identify
$\Z^\Gamma$ with the root lattice $X(\Gamma)$ by sending $\bd\mapsto \nu=\sum d_i\al_i$.

Let $\Bi$ be a loading.
\begin{definition}
  We let an $\Bi$-loaded flag on $V$ be a flag of $I$-homogeneous
  subspaces $F_a\subset V$ for each real number $a$ such that
  $F_b\subset F_a$ for $b\leq a$, and $\dim F_a=\sum_{b\leq a}
  \Bi(b)$.  Even though this filtration is indexed by real numbers,
  only finitely many different spaces appear; the dimension vector can only change at points in the
  support of the loading, by adding the simple root labeling that
  point to the dimension vector.  Let $\Fl_\Bi$ denote the space of $\Bi$-loaded flags.
\end{definition}

The relationship of these flags to the loadings we discussed earlier
(justifying the name) is as follows:  we can imagine the space $F_a$
as being attached to the dots left of $x=a$.  We read from left to
right, and each time we pass a dot with label $i$, we increase the
size of the space in the flag in $V_i$.

Each loaded flag $F_\bullet$ has a corresponding {\bf unloading},
which is the complete flag of spaces appearing as $F_a$ for $a\in \R$,
indexed by dimension as usual.

\begin{definition}
  For $\Bi$ a loading with $|\Bi|=\nu$,
  let \[X_{\Bi}=\{(f,F_\bullet)\in E_{\nu}\times \Fl_\Bi |
  f_e(F_a)\subset F_{a-\vartheta_e} \}\]  be the space of $\Bi$-loaded
  flags and compatible representations.  Let $p\colon X_{\Bi}\to
  E_\nu$ be the map forgetting the flags, and
  let \[Z=\bigsqcup_{\Bi,\Bj\in B(\nu)}X_{\Bi}\times_{E_\nu}X_{\Bj}.\]
\end{definition}
We can also interpret compatibility visually in terms of loadings:
rather than require that $F_a$ be preserved by $f_e$, we require that
the piece of $V_i$ corresponding to a dot at $x=a-\vartheta_e$ can
only be hit under the map $f_e$ by the pieces corresponding to dots
right of the corresponding ghost, that with $x\geq a$.  Put
differently, the piece of the filtration $F_a$ corresponding to dots
left of $a$ must land under $f_e$ in the span of pieces for dots whose
ghosts are left of $x=a$.   

\begin{example}
  For $\Bj(\hat{\bmu})$ with $k>0$, as defined in Section \ref{sec:relat-quiv-schur},
  the map $f_e$ for each edge $e\colon i\to i+1$ must send the
  $F_{js}$ space
  associated to the first $j$ parts of the vector composition to the
  space $F_{(j-1)s}$ for the $j-1$ pieces, since we have specifically
  set things up so that a dot in the $j$th piece is to the left of the
  ghosts attached to the $j$th piece, and those to the right, and
  right of the ghosts for the $j-1$st piece, and those to the left.
  Note that this is closely related to the flag spaces considered in
  \cite{SWschur}, where arbitrary strongly preserved flags were
  considered, but the flags we consider here come with a refinement to
  complete flags.  While this may seem extraneous, it makes the
  convolution algebras much easier to deal with.

If $k<0$, then the picture is quite different.  Now, each dot for the
$j$th piece is right of the dots in the $j$th piece (and those to the
left), so our conditions just say that $f_e(F_{js})\subset F_{js}$, so
this flag is weakly preserved.  
\end{example}

\begin{example}
  If we consider a Crawley-Boevey quiver $\Gamma_\Bw$, with the weight
  on all old edges being 0, then the result is that the flag $F_a$
  must be preserved in the usual sense by all the maps associated to
  old edges.  Furthermore, the map $f_e$ along a new edge is constrained
  to be 0 on $F_{\vartheta_e}$. That is, we are only allowed to use
  one of the new edges on pieces of the flag corresponding to dots
  coming right of the corresponding red line (in the usual pictures
  discussed in Section \ref{sec:relat-tens-prod}.
\end{example}

If $\vartheta_e=0$, then these are simply {\bf quiver flag varieties},
as used by Lusztig \cite{Lus91}, and considered by many other authors since.
In particular, we can define a collection of objects in the
$G_\nu$-equivariant derived category of $E_\nu$ generalizing those
considered by Lusztig, by considering the pushforwards
\[Y_\Bi:=p_*\K_{X_{\Bi}}[\Bu(\Bi)]\] where \[\Bu(\Bi)=\dim
X_\Bi/G_\nu=\#\{ (e,a,b)\mid \Bi(a)=t(e),\Bi(b)=h(e), a-b\geq
\vartheta_e\}-\sum_{i\in I}\frac{|\Bi|_i(|\Bi|_i+1)}{2}.\]  
Since $p$
is proper, if $\K$ is characteristic 0, then these sheaves will be a sum of
shifts of simple perverse sheaves; this can fail when the
characteristic is positive and small.  In favorable cases, where we
obtain parity vanishing results, the
summands of these sheaves will be parity sheaves in the sense of
Juteau, Mautner and Williamson \cite{JMW}.  This is the case when
$\Gamma$ is of finite or affine type A, but seems to be unknown in
general; see \cite{Makpar} for a more detailed discussion of parity
sheaves on $E_{\Bv}$.  

In the case of a tensor product weighting, these spaces and sheaves
have been studied by Li \cite{Litensor}. In the affine case, closely
related spaces were considered in \cite{SWschur}; as long as the
weights on new edges are well separated, the sheaves $Y_\Bi$ have the
same simple summands as the pushforwards from the spaces
$\mathfrak{Q}(\hat{\bmu})$.  This definition of the spaces $X_\Bi$ has
motivated in large part by the desire to unify these examples and put
them in a more general context.

\subsection{An Ext-algebra calculation}
\label{sec:an-ext-algebra}
Consider the tautological line bundle
$\mathscr{L}_i$ given by the quotient of the $i$-dimensional space of
the flag by the $i-1$st. 
The cohomology ring
  $H^*_{G_\nu}(\Fl_\Bi)$ is a polynomial ring, in variables that can be
  identified with the equivariant Chern classes $\mathscr{L}_k $.

Given two loadings $\Bi$ and $\Bj$ and a permutation $\si$, we have a natural correspondence
\[\becircled{X}_{\Bi;\Bj}^{\tau}=\{(f,\{F_\bullet\},f',\{F_\bullet'\})\in X_\Bi\times
X_\Bj| r(V_*,V_*')= \tau\} \qquad X^{\tau}_{\Bi;\Bj}=\overline{\becircled{X}_{\Bi;\Bj}^{\tau}}\] where $r(-,-)$ is the usual relative
position between the unloadings of these flags.  This space is
non-empty if and only if the unloadings of  $\Bi$ and $\Bj$ are permuted to each other by $\tau$. 

We let $H^{BM,G_{\bd}}_*(-)$
denote the equivariant Borel-Moore homology of a space with coefficients in $\K$, as discussed
in \cite[\S 1]{VV}; for any proper map $p\colon X\to Y$, the Borel-Moore
homology $H^{BM}(X\times_Y X)$ 
carries a convolution algebra structure, defined in \cite[2.7]{CG97};
in  \cite[8.6]{CG97}, it's proven that this is isomorphic to the Ext
algebra $\Ext^\bullet
  (p_*\K_X,p_*\K_X)$, and this result is easily extended to the
  equivariant case.
\begin{theorem}\label{thm:R-ext}
    We have isomorphisms of dg-algebras \[\Ext^\bullet_{G_\nu}
  \Big(\bigoplus_{\Bi\in B(\nu)}Y_\Bi,\bigoplus_{\Bi\in
    B(\nu)}Y_\Bi \Big)\cong H^{BM,G_{\nu}}_*(Z )\cong
  \walg^\vartheta_\nu\] where the RHS has trivial differential.  The
 right hand isomorphism sends \[e_{\Bi}b_{1}e_{\Bj}\mapsto
 [X^{1}_{\Bi;
\Bj}]\qquad e_{\Bi}\psi_ke_{\Bj}\mapsto
 [X^{s_k}_{\Bi;
\Bj}]\qquad y_k\to c_1(\mathscr{L}_k).\]
This map intertwines Verdier duality and the duality $a\mapsto a^*$ on $\walg^\vartheta_\nu$.
\end{theorem}
\begin{remark}
  If the characteristic of $\K$ is positive, then this result is still true
  as an isomorphism of algebras, but it seems unlikely that the dg or $A_\infty$ structure on
  the left hand side is formal.  
\end{remark}
Recall that replacing an object by another in which precisely the same
indecomposable summands occur preserves the graded Morita class of the
Ext-algebra.  Thus, if we replace $\oplus Y_\Bi$ by the sum of all
IC-sheaves whose shifts appear as summands of $Y_\Bi$ for some $\Bi$,
we obtain that:
\begin{corollary}\label{cor:positive-grading}
  The algebra $\walg^\vartheta$ is graded Morita equivalent to a non-negatively
  graded algebra, with semisimple degree 0 subalgebra.  That is, there is a projective generator $G$ of
  $\walg^\vartheta\mmod$ with no negative degree endomorphisms, and
 all degree 0 endomorphisms spanned by projection to the different summands.   We can choose
  this generator so that if $P$ is a
  graded projective that occurs as a summand in $\walg^\vartheta
  e_\Bi$ for some $\Bi$ such that no shift of $P$ does, then $P$ is a
  summand of $G$.
\end{corollary}
Note that it is easy to find examples where this fails if $\K$ has
characteristic $p$.  Such an example for $\mathfrak{\widehat{sl}}_2$
is discussed in \cite[5.7]{WebCB}; Williamson \cite{WilJames} has shown that examples
exist for KLR algebras in finite type A for any prime $p$.  As we see
in \cite{WebCB,WebRou}, this property is key for proving a
relationship between categorifications and canonical bases, along the
same lines as \cite{VV}.

We now turn to the proof of Theorem \ref{thm:R-ext}, which we will
prove via a series of lemmata.  As we noted in the
proof of \cite[3.11]{SWschur}, we have an equivariant map
\[X_\Bi\times_{E_{\nu}}X_\Bj\to \Fl_{\Bi}\times \Fl_{\Bj},\] projecting to the second factor.   This map is is an affine
bundle over each $G_\nu$-orbit.
These orbits are in turn homotopic to $G_\nu/T_\nu$, letting $T_\nu$ be a
maximal torus in $G_\nu$. Thus
$X_\Bi\times_{E_{\nu}}X_\Bj$ is a union of finitely
many spaces each with even and equivariantly formal Borel-Moore homology,
so the same is true of $X_\Bi\times_{E_{\nu}}X_\Bj$.

\begin{lemma}
  The Ext-algebra $E=\Ext^\bullet_{G_\nu}
  \Big(\bigoplus_{\Bi\in B(\bd)}Y_\Bi,\bigoplus_{\Bi\in B(\bd)}Y_\Bi
  \Big)$ is formal and acts faithfully on
  \[\bigoplus_{\Bi\in B(\nu)}H^*_{G_\nu}(X_\Bi)\cong
  \bigoplus_{\Bi\in B(\nu)}H^*_{G_\nu}(\Fl_\Bi).\]
\end{lemma}

  \begin{proof}
By a result of Lunts \cite[6.2]{LuntsDG} based on work of Kaledin,
formality is unchanged by extending base field, so it suffices to prove
this formality for $\K$ a single characteristic 0 field.   If $\K=\C$,
then the algebra $ H^{BM,G_{\nu}}_*(Z )$ has a Hodge structure.   The subset of
$Z$ where we fix the relative position of the two flags, and the
Schubert cell the left flag lies is an iterated affine bundle, and
thus isomorphic to affine space.  Since
$Z$ is a union of finitely many algebraic cells, the Hodge structure
on $H^{BM,G_{\nu}}_*(Z )$ is pure.  All $A_\infty$ operations are
compatible with this Hodge structure, so purity implies that they are
homogeneous of degree 0 in the homological grading.  On the other
hand, the $A_\infty$ operation $m_k$ is homogenous of degree $2-k$,
implying that it is 0 unless $k=2$, so this $A_\infty$ structure is formal.

   The proof of faithfulness is essentially the same as \cite[4.7]{SWschur}.
 Let $U=H^*(BG_\nu)$ and $V=H^*(BT_\nu)$.  The restriction functor
    $\operatorname{Rest}^{G_\nu}_{T_\nu}$ on equivariant derived
    categories and the
    inclusion
    $\iota_{\Bi,\Bj}:(X_\Bi\times_{E_{\nu}}X_\Bj)^{T_\nu\times
      T_\nu}\hookrightarrow X_\Bi\times_{E_{\nu}}X_\Bj$ induce a
    commutative diagram
 \[\tikz[->,thick]{\matrix[row sep=3mm,column
      sep=10mm,ampersand replacement=\&]{ \& \node (b) {$
          \Hom_{U}(H^*_{G_\nu}(X_{\Bi} ),H^*_{G_\nu} (X_{\Bj}))$};\\
        \node(a) {$H^{BM,G_\nu}_*(X_{\Bi}\times_{E_\nu}X_{\Bj})$}; \&\\
        \& \node (g) {$
          \Hom_{V}(V\otimes_{U} H^*_{G_\nu}(X_{\Bi} ),V\otimes_{U} H^*_{G_\nu} (X_{\Bj} ))$};\\
        \node(c) {$H^{BM,T_\nu\times
            T_\nu}_*(X_{\Bi}\times_{E_\nu}X_{\Bj})$}; \&\\ \&
        \node(d){$
          \Hom_{V}(H^*_{T_\nu} (X_{\Bi}),H^*_{T_\nu} (X_{\Bj}))$};\\
        \node(e) {$H^{BM,T_\nu\times
            T_\nu}_*((X_\Bi\times_{E_{\nu}}X_\Bj)^{T_\nu\times
            T_\nu})$}; \&\\ \& \node(f){$
          \Hom_{V}(H^*_{T_\nu}(X_\Bi^{T_\nu}),H^*_{T_\nu} (X_\Bj^{T_\nu})) $};\\
      }; \draw (a)-- node [above,midway]{$\star-$} (b); \draw (c)--
      node [above,midway]{$\star-$} (d); \draw (e)-- node
      [above,midway]{$\star-$} (f); \draw (a)-- node
      [above,left]{$\operatorname{Rest}^{G_\nu}_{T_\nu}$} (c); \draw
      (c)--node [left,midway]{$\iota^*_{\Bi,\Bj}
        (\iota_{\Bi,\Bj})_*\iota^*_{\Bi,\Bj}$} (e); \draw (b)-- node
      [right,midway]{$\operatorname{id}_V\otimes -$} (g); \draw (g)--
      node [above,midway,rotate=-90]{$\sim$} (d); \draw (d)-- node
      [right,midway]{$\iota^*_{\Bj}\circ -\circ (\iota_\Bi)_*$}(f);
    }\]
   The composition of the two vertical lines are both injective, since
   $V$ is a free module of finite rank over $U$ and the Borel-Moore
   homology of every space that appears is even and equivariantly
   formal. Furthermore, the bottom rung of the ladder is injective.  Thus, any class $a\in
    H^{BM,G_\nu}_*(X_{\Bi}\times_{E_\nu}X_{\Bj})$ which the top action
    kills is also killed by the map from the northwest corner to the
    southeast.  This map is injective, so we are done.
  \end{proof}

  \begin{lemma}
    The non-zero classes $[X^{\sigma}_{\Bi,\Bj}]$ are a
    basis of $H^{BM,G_{\nu}}_*\big(X_\Bi\times_{E_\nu}X_\Bj\big)$
    over $H^*_{G_\nu}(\Fl_{\Bi})$.
  \end{lemma}
  \begin{proof}
   Pick a total order on permutations refining Bruhat order; our inductive
    statement is that $[X^{\sigma}_{\Bi,\Bj}]$ for $\si\leq \tau$ is
    a basis of $H^{BM,G_{\bd}}_*(\cup_{\si\preceq \tau}X_{\Bi;\Bj}^{\si})$.  If
    $\tau=1$, then $X_{\Bi;\Bj}^{1}$ is an affine bundle over
    $\Fl_\Bi=\Fl_{\Bj}$, since the left and right flags must agree.
    Thus, its equivariant Borel-Moore homology is freely generated over
    $H^*(\Fl_{\Bi})$ by $[X^{1}_{\Bi,\Bj}]$.

Now, by induction, let $\tau'$ be maximal w.r.t $\tau'\prec \tau$.
Then we have long exact sequence
\[\cdots \longrightarrow H^{BM,G_\bd}_i(\cup_{\si\preceq
  \tau'}X_{\Bi;\Bj}^{\si})\to H^{BM,G_\bd}_i(\cup_{\si\preceq
  \tau}X_{\Bi;\Bj}^{\si})\to
H^{BM,G_\bd}_i(\becircled{X}_{\Bi;\Bj}^{\tau})\to \cdots\] The space
$\becircled{X}_{\Bi;\Bj}^{\tau} $ is an affine bundle over the space
in $\Fl_{\Bi}\times \Fl_{\Bj}$ with relative position $\tau$, since
being compatible with two fixed flags is a linear condition on matrix
coefficients of quiver representations, and all fibers are conjugate under
the action of $G_\nu$. This space is, in turn, an affine bundle over $\Fl_{\Bi}$
since the space of flags of relative position exactly $\tau$ to a
fixed flag is an affine space. Thus, the equivariant Borel-Moore homology $ 
H^{BM,G_\bd}_i(\becircled{X}_{\Bi;\Bj}^{\tau})$ is
free of rank 1 over $H^*(BG_{\Bi})$ if the unloading of $\Bi$ is sent
to the unloading of $\Bj$ by $\tau$, and rank 0
otherwise (since the space is empty). Furthermore, it is generated by the fundamental class of
$\becircled{X}_{\Bi;\Bj}^{\tau}$ and in particular all lies in even
degree.  This shows that, by induction, all groups appearing in the
above sequence vanish in odd degree, so the long exact sequence splits into a sum
of short exact sequences.

Thus, any subset of $ H^{BM,G_\bd}_i(\cup_{\si\preceq
  \tau}X_{\Bi;\Bj}^{\si})$ consisting of a basis of
$H^{BM,G_\bd}_i(\cup_{\si\preceq \tau'}X_{\Bi;\Bj}^{\si})$ and an
element projecting to $[\becircled{X}_{\Bi;\Bj}^{\tau}]$ (if that
space is non-empty) is a basis of $ H^{BM,G_\bd}_i(\cup_{\si\preceq
  \tau}X_{\Bi;\Bj}^{\si})$.  Since $[{X}_{\Bi;\Bj}^{\tau}]$ projects to
$[\becircled{X}_{\Bi;\Bj}^{\tau}]$ if that class is non-zero, and is
itself 0 otherwise, induction yields the desired fact.
  \end{proof}

\begin{proof}[Proof of Theorem~\ref{thm:R-ext}]
First, the left hand isomorphism is an immediate consequence of \cite[8.6.7]{CG97}.

Now we wish to confirm that the action of the classes
$[X^{1}_{\Bi;\Bj}]$ and $[X^{s_k}_{\Bi;\Bj}]$ act on
\[\bigoplus_{\Bi\in B(\nu)}H^*_{G_\nu}(X_\Bi)\cong \bigoplus_{\Bi\in
  B(\nu)}\K[y_1,\dots,y_d]\] in the same way as $e_{\Bi}b_1e_{\Bj}$ and $e_{\Bi}\psi_ke_{\Bj}$.

\begin{itemize}
\item If going from $\Bi$ to $\Bj$ passes a strand from right of a ghost to left of it, then $X^{1}_{\Bi;\Bj}\cong X_{\Bj}$: any  $\Bj$-loaded flag is easily modified to be a $\Bi$-loaded flag using reindexing.  Thus, the desired convolution is just the pull-back map for the inclusion $X_{\Bj}\to X_{\Bi}$ in cohomology, which sends Chern classes to Chern classes, and induces the identity map on $\C[y_1,\cdots,y_n]$. 
\item If going from $\Bi$ to $\Bj$ passes the $j$th strand from left of a ghost for $e$ of the $k$th strand to right, then symmetrically $X^{1}_{\Bi;\Bj}\cong X_{\Bi}$.  Thus, the desired convolution is the pushforward by the inclusion $X_{\Bi}\to X_{\Bj}$, which on the level of cohomology rings multiplies by the Euler class of the normal bundle for the inclusion, which is $\Hom(\mathscr{L}_j,\mathscr{L}_k)$, whose Euler class is $y_k-y_j=Q_e(y_k,y_j)$.
\end{itemize}
This deals with all crossings of strands and ghosts.  We now need only consider the case where no ghosts separate the $k$ and $k+1$st strands, and we apply $\psi_k$.
\begin{itemize}
\item If $k$th and $k+1$st strands have different labels, then $X^{s_k}_{\Bi;\Bj}$ is the graph of an isomorphism between the sets of loadings $X_{\Bi}$ and $X_{\Bj}$; there is a unique $\Bj$-loaded flag which agrees with a given $\Bi$-loaded flag at all jumps but the $k$th.  The only effect of this isomorphism is that it reindexes the line bundles of interest to us via the permutation $s_k$; hence this is also the effect on Chern classes.
\item If the $k$th and $k+1$st strands have the same labels, we can take $\Bi=\Bj$. Let $W$ be the subvariety of $X_{\Bi}$ where all loops of weight $0$ send the $k+1$st step of the flag to the $k-1$st, and let $\mathscr{L}_{k;k+1}$ be the rank 2 vector bundle on $W$ given by the $k+1$st  step of the flag modulo the $k-1$st.  The space $X^{s_k}_{\Bi;\Bi}$ is the projectivization over $W$ of the vector bundle $\mathscr{L}_{k;k+1}$.  Thus, if $i\colon W\to X_{\Bi}$ is the inclusion, and $p\colon X^{s_k}_{\Bi;\Bi}\to W$ the projection, then $[X^{s_k}_{\Bi;\Bi}]=i_*p_*p^*i^*$.  The two pullbacks just act as the identity; the pushforward $p_*$ acts as Demazure operator in the variables $y_k$ and $y_{k+1}$, and the pushforward $i_*$ multiplies by the Euler class of the normal bundle, which is 1 if there is no loop of degree 0, and $\prod_eQ_e(y_k,y_{k+1})$ where $e$ ranges over such loops otherwise.  Applying Definition \ref{def:arbitrary}, we see that this matches the action of Proposition \ref{prop:action}.
\end{itemize}
This shows that we have an injective algebra map $a\colon
e_{\Bi}\walg^\vartheta_\nu e_{\Bj}\to
H^{BM,G_{\nu}}_*\big(X_\Bi\times_{E_\nu}X_\Bj\big)$.  Finally, we need
to confirm that this map is surjective.

We let $e_{\Bi}x_\sigma e_{\Bj}=[X^{\sigma}_{\Bi,\Bj}]$ if the word in simple
roots attached to $\Bj$ is the permutation by $\sigma$ of that for
$\Bi$ and 0
otherwise.

Now, consider a factorization of $b_\tau$ into pieces where there is
only one crossing of two strands or of a strand and a ghost. The image
$a(b_\tau)$ of this diagram is the convolution of all the classes
attached to these diagrams, which are each of the form
$[X^{s_k}_{-,-}]$ or $[X^{1}_{-,-}].$ That is, there is sequences $t_m\in \{1\}\cup \{s_1,\dots, s_n\}$ and $\Bi^{(m)}$ such that $b_\tau=e_{\Bi}b_{t_1}e_{\Bi^{(1)}}b_{t_2}\cdots e_{\Bi^{(\ell-1)}}b_{t_\ell}e_{\Bj}$ In particular we obtain a reduced
decomposition $\tau=t_{k_1}\cdots t_{k_\ell}$. Now, consider an element
$(f,F_\bullet, F_\bullet')\in \becircled{X}^{\tau}_{\Bi,\Bj}$. Consider the unique flag which has relative position $t_{k_1}\cdots t_{k_h}$
to the unloading of the left flag and $t_{k_{\ell}}\cdots t_{k_{h+1}}$
to the unloading of the right.
Let $F^h_\bullet$ be the unique $\Bi^{(h)}$-loaded flag 
whose unloading is the complete flag we have just described.
\begin{lemma} 
The $\Bi^{(h)}$-loaded flag $F^h_\bullet$ is compatible
  with the representation $f_*$. 
\end{lemma}

\begin{proof} 
Without loss of generality, we can assume that both
  $F_\bullet$ and $F_{\bullet}'$ are coordinate flags for a single
  basis, which is in bijection with the points in the loadings $\Bi$
  and $\Bj$; we let $w_i$ and $w_j$ be the accompanying positions on
  the real line. By the compatibility with $F_\bullet$ and
  $F_\bullet'$, the image $f_e(v_m)$ is in the span of $v_k$ with
  $w_i(v_k) \leq w_i(v_m)-\vartheta_e$ and $w_j(v_k) \leq
  w_j(v_m)-\vartheta_e$.

One of the essential characteristics of $b_\tau$ is that up to isotopy, we can assume that the distance between any pair of strands monotonically increases or decreases, so we may assume that the distance between the weights associated to $v_k$ and $v_m$ in $\Bi^{(h)}$ are strictly between that for $\Bi$ and $\Bj$.  Thus, the same inequalities hold for every slice, and we are done.
\end{proof}

Thus, we see that the map from the fiber product \[q\colon
X^{t_1}_{\Bi;\Bi^{(1)}}
\times_{X_\Bi^{(1)}}X^{t_2}_{\Bi^{(1)};\Bi^{(2)}} \times
_{X_\Bi^{(2)}}\cdots \times_{X_\Bi^{(\ell-1)}}
X^{t_\ell}_{\Bi^{(\ell-1)};\Bj}\to X^{\tau}_{\Bi;\Bj}\] must map
bijectively over $\becircled{X}^\tau_{\Bi;\Bj}$; at each intermediate point, we have a single unique choice for the $\Bi^{(h)}$-loaded flag compatible with $f_*$, which is, of course, $F^h_*$.

Thus, we have that \[a(b_\tau)=q_*[ X^{t_1}_{\Bi;\Bi^{(1)}} \times_{X_\Bi^{(1)}}\cdots \times_{X_\Bi^{(\ell-1)}} X^{t_\ell}_{\Bi^{(\ell-1)};\Bj}]=[X^\tau_{\Bi;\Bj}]+\sum_{\tau<\tau} r_\tau(y_1,\dots, y_n) [X^\tau_{\Bi;\Bj}] .\]

Thus, the matrix of the map $a$ written in terms of the basis of $e_\Bi \walg^\vartheta
e_{\Bj}$ given by $b_\tau$'s and that for the Borel-Moore homology $H^{BM,G_{\nu}}_*\big(X_\Bi\times_{E_{\nu}}X_\Bj\big)$ given by $[X^\tau_{\Bi;\Bj}]$'s is upper-triangular with 1's on the diagonal and thus an isomorphism.
\end{proof}

Put another way:
\begin{corollary}
  There is a fully faithful additive functor $\gamma\colon
  \walg^\vartheta_\nu\pmmod\to D(E_\nu /G_\nu)$ sending $[\walg^\vartheta_\nu
  e_{\Bi}]\mapsto Y_\Bi$.
\end{corollary}
If $\Gamma$ is produced by the Crawley-Boevey trick on another graph,
we let $G_\nu'$ be the subgroup of $G$ which only acts on old
vertices.  This is a codimension 1 subgroup.  

If we let $Y_\Bi'$ be the pullback of $Y_\Bi$ from $E_\nu/G_\nu$ to
$E_\nu/G_\nu'$.  Repeating the proof of Theorem \ref{thm:R-ext} in
this context, we arrive at almost the same result, except that we have
killed the Chern class of any line bundle attached to a representation
which is trivial restricted to $G_\nu'$, that is the dot on the unique
strand labeled with $\al_0$.  That is:
\begin{corollary}
  We have isomorphisms of dg-algebras \[\Ext^\bullet_{G_\nu}
  \Big(\bigoplus_{\Bi\in B(\nu)}Y_\Bi',\bigoplus_{\Bi\in
    B(\nu)}Y_\Bi' \Big)\cong \bigoplus_{\Bi,\Bj\in B(\nu)}H^{BM,G_{\nu}'}_*\big(X_\Bi\times_{E_{\nu}}X_\Bj\big)\cong
  \bar \walg^\vartheta_\nu\] where the RHS has trivial differential.   In
  particular, if we choose a tensor product weighting, we have an
  isomorphism \[\Ext^\bullet_{G_\nu}
  \Big(\bigoplus_{\Bi\in B(\nu)}Y_\Bi',\bigoplus_{\Bi\in
    B(\nu)}Y_\Bi' \Big)\cong
  \tilde T^\bla_{\la-\nu}\]
\end{corollary}

This result naturally extends to the bimodule
$B^{\vartheta,\vartheta'}$ defined earlier.  The proof is so similar
to that of Theorem \ref{thm:R-ext} that we leave it to the reader:

\begin{theorem}\label{B-convolution}
  For two weightings $\vartheta_1,\vartheta_2$, we have an isomorphism
  of dg-modules:
 \[\Ext^\bullet_{G_\nu}
  \Big(\bigoplus_{\Bi\in B^1(\nu)}Y_\Bi^1,\bigoplus_{\Bj\in
    B^2(\nu)}Y_\Bj^2 \Big)\cong
  B^{\vartheta_1,\vartheta_2}_\nu\] where the left and right algebra
  actions are matched using the isomorphism of Theorem \ref{thm:R-ext}. 
\end{theorem}

\begin{remark}
  Theorems \ref{thm:R-ext} and \ref{B-convolution} can be
  extended to the canonical deformations of Section
  \ref{sec:canon-deform} by letting $\mathbb{G}_m^{E(\Gamma)}$ act in
  the natural way on $E$ with each copy of $\mathbb{G}_m$ acting with
  weight 1 on the map along one edge and trivially on all others.
  Considering the equivariant Borel-Moore homology in place of usual
  BM homology gives the deformed algebra $\dwalg^\vartheta$, with the
  deformation parameters corresponding to the cohomology of $B
  \mathbb{G}_m^{E(\Gamma)}$.
\end{remark}

\subsection{Monoidal structure}
\label{sec:monoidal-structure}

Recall that the derived categories $\oplus_\nu D(E_\nu/G_\nu)$ carry
the Lusztig monoidal structure defined by convolution.  If
$\nu=\nu'+\nu''$, and we let $V_i=V_i'\oplus V_i''$ be $I$-graded
vector spaces of the appropriate dimension, we consider
\[E_{\nu';\nu''}\cong 
E_{\nu'}\oplus E_{\nu''} \oplus\bigoplus_{e\in \Omega} \Hom(V''_{t(e)},V'_{h(e)})\]
with the obvious action of \[G_{\nu';\nu''}=\{g\in G_{\nu}|
g(V_i')=V_i'\}.\]  We have the usual convolution diagram 
\[\tikz[very thick,->]{\matrix[row sep=11mm,column sep=16mm,ampersand
    replacement=\&]{
\& \node (a) {$E_{\nu';\nu''}/G_{\nu';\nu''}$}; \&  \\
\node (b) {$E_{\nu'}/G_{\nu'}$};\& \node (c) {$E_{\nu}/G_{\nu}$}; \&  \node (d) {$E_{\nu''}/G_{\nu''}$}; \\
};
\draw (a)--(b) node[above left,midway]{$\pi_s$};
\draw (a)--(c)  node[left,midway]{$\pi_t$};
\draw (a)--(d)  node[ above right,midway]{$\pi_q$};
}\]
We can view $E_{\nu';\nu''}/G_{\nu';\nu''}$ as the moduli space of
short exact sequence with submodule of dimension $\nu'$ and quotient
of $\nu''$.  The projections $\pi_*$ are remembering only the first,
second or third term of the short exact sequence.
The {\bf convolution} of sheaves $\mathcal{F}_1\in
D(E_{\nu'}/G_{\nu'}), \mathcal{F}_2\in
D(E_{\nu''}/G_{\nu''})$ is defined to be 
\[\mathcal{F}_1 \star \mathcal{F}_2:=
(\pi_t)_*(\pi_s^*\mathcal{F}_1\otimes \pi_t^*\mathcal{F}_2)[-\langle
\nu'',\nu'\rangle]\]

\begin{proposition}\label{pro:monoidal}
  The functor $\gamma \colon \walg^\vartheta_\nu\pmmod\to
  D(E_{\nu}/G_{\nu})$ is monoidal, i.e. \[\gamma(P_1\circ P_2)\cong
  \gamma(P_1)\star \gamma(P_2).\]
\end{proposition}
\begin{proof}
  We need only check this for $P_1= \walg^\vartheta_{\nu'}e_{\Bi},P_2=
  \walg^\vartheta_{\nu''}e_{\Bj}$ since every projective is a summand of
  one of these.  In this case, $P_1\circ
  P_2=\walg^\vartheta_{\nu}e_{\Bi\circ \Bj}$.  On the other hand,
  \[\pi_{s}^*Y_{\Bi}=\tilde
  p_*^s\K_{X_{\Bi}\times_{E_{\nu'}}E_{\nu';\nu''}}[\Bu(\Bi)]\qquad
    \pi_{q}^*Y_{\Bj}=\tilde
    p_*^q\K_{X_{\Bj}\times_{E_{\nu''}}E_{\nu';\nu''}}[\Bu(\Bi)]\] 
where $\tilde p_*^s$ and $\tilde p_*^q$ are base changes of the map
      $p$ by $\pi_s$ and $\pi_q$.  Note that when $\Bi$ and $\Bj$ are separated far enough that no
ghost from one is entangled in the other, the subspace $F_a$ for $a$
between $\Bi$ and $\Bj$ on the real line is a subrepresentation.
Thus we have an isomorphism
\begin{equation}
(X_{\Bi}\times X_{\Bj})\times_{E_{\nu'}\times
        E_{\nu''}}E_{\nu';\nu''}/G_{\nu';\nu''}\cong X_{\Bi\circ
        \Bj}/G_{\nu};\label{eq:2}
    \end{equation}
the
difference in groups is that on left side we fix a particular subspace
and assume $F_a=\oplus V_i'$ and only act with the stabilizer of this
subspace, whereas on the right side, we sweep through all possible
subspaces. These quotients are the same since all $I$-graded
subspaces of the same dimension vector are conjugate under $G_\nu$.

 By definition, $Y_\Bi\star Y_{\Bj}$ is the
      shifted pushforward from the LHS of \eqref{eq:2}, and $Y_{\Bi\circ \Bj}$ is the
      shifted pushforward from the RHS.
Thus we have that $Y_\Bi\star Y_{\Bj}\cong Y_{\Bi\circ \Bj}$ where the
equality of shifts follows from the formula
\[\Bu(\Bi\circ \Bj)=\Bu(\Bi)+\Bu(\Bj)- \langle|\Bj|,|\Bi| \rangle.\]
Furthermore, the self-Exts of $Y_\Bi\star Y_{\Bj}$ induced by those of
$Y_\Bi$ and $Y_\Bj$ are exactly intertwined with the image of
$\iota_{\nu';\nu''}$, which shows that this functor is monoidal on
morphisms as well.  Thus, we have obtained the desired result.
\end{proof}

There is also a left adjoint to $\star$, which we denote
$\Res_{\nu';\nu''}$, given by 
\[\Res_{\nu';\nu''}\mathcal{F}:=
(\pi_s\times \pi_q)_! \pi_t^!\mathcal{F}[\langle
\nu'',\nu'\rangle]\]
\begin{proposition}\label{prop:co-monoidal}
  The functor $\gamma \colon \walg^\vartheta_\nu\pmmod\to
  D(E_{\nu}/G_{\nu})$ is intertwines restriction
  functors, that is \[(\gamma\boxtimes \gamma) (\Res_{\nu';\nu''}P)\cong
 \Res_{\nu';\nu''} \gamma(P).\]
\end{proposition}
\begin{proof}
  Since these functors are left adjoint to functors intertwined by
  $\gamma$, they just be intertwined if $\Res_{\nu';\nu''} \gamma(P)$
  is in the subcategory generated by the image of $\gamma(P)$.

  As before, we need only consider the base where $P=Re_{\Bi}$.  In
  this case, $\pi_t^!Y_{\Bi}=\tilde p_*\K_{X_\Bi\times
    E_{\nu';\nu''}}$.  We filter the fiber product $X_\Bi\times
    E_{\nu';\nu''}$ according the relative position of the subspace
    $V_i'$ and the $\Bi$-loaded flag (i.e. by the Schubert cell $V_i'$
    lands in for the Schubert stratification relative to the flag).
    Each such relative position corresponds to dividing the points in
    the loading into two sets: those where the dimension of the
    intersection of $F_a$ with $V_i'$ jumps and those where it does
    not.   This gives loadings $\Bi'$ and $\Bi''$.  The subset of the fiber product $X_\Bi\times
    E_{\nu';\nu''}$ with fixed relative position is an affine bundle
    over the product $X_{\Bi'}\times X_{\Bi'}$ where the first term is
    the loaded flag induced on $V_i'$ by intersecting with $F_a$, and
    the second is that induced on $V_i''$ by taking the images of the
    $F_a$'s.  This shows that   $\Res_{\nu';\nu''} \gamma(P)$ is an
    iterated cone of shifts of the objects $\gamma(P')\boxtimes
    \gamma(P'')$.  This completes the proof.
\end{proof}

On the level of Grothendieck groups, these propositions show that the
structures we have seen on $K$ are also typical for categories of
sheaves on representations of quivers.  

\begin{proposition}\label{pro:sheaf-bialgebra}
  The sum of Grothendieck groups $\oplus_\nu K(D(E_\nu/G_\nu))$ inherits a
  twisted bialgebra structure with product and coproduct
\[[\cM][\cN]=[\cM\star \cN]\qquad \Delta([\cM])=\left[\sum_{\nu'+\nu''=\nu}\Res_{\nu';\nu''}\cM\right],\] and the functor $\gamma$ induces a map
  of twisted bialgebras.
\end{proposition}
\begin{proof}
The fact that $\gamma$ induces a map that commutes with the
multiplication and comultiplication follows from Propositions
\ref{pro:monoidal} and \ref{prop:co-monoidal}.

  The commutation of product and coproduct follows from the base change
  formula for pushforwards and pullbacks.  Choosing
  $\nu',\nu'',\mu',\mu''$ such that $\nu'+\nu''=\nu=\mu'+\mu''$, we
  wish to consider $\Res_{\mu',\mu''}(\cM'\star \cM'')$.  Let $\pi_*$
  denote the projection maps from $E_{\nu';\nu''}$ as before and
  $\kappa_*$ the corresponding maps from $E_{\mu';\mu''}$ and
  $B=E_{\nu';\nu''}\times_{E_\nu} E_{\mu';\mu''}$.  Then we have a
  diagram with the interior square Cartesian:
\[\tikz[very thick,->]{\matrix[row sep=11mm,column sep=7mm,ampersand
    replacement=\&]{
\&\& \node (a) {$B/G_\nu$}; \&  \&\\
\&\node (b) {$E_{\nu';\nu''}/G_{\nu}$};\& 
\&  \node (c) {$E_{\mu';\mu''}/G_{\nu}$};\& \\
\node (d) {$E_{\nu'}/G_{\nu'}\times E_{\nu''}/ G_{\nu'}$};\& \&\node
(e) {$E_\nu/G_\nu$};\&\&\node (f) {$E_{\mu'}/G_{\mu'}\times E_{\mu''}/ G_{\mu'}$};\\
};
\draw (a)--(b) node[above left,midway]{$\tilde \kappa_t$};
\draw (a)--(c)  node[above right,midway]{$\tilde \pi_t$};
\draw (b)--(d)  node[ above left,midway]{$\pi_s\times \pi_q$};
\draw (b)--(e) node[below left,midway]{$\pi_t$};
\draw (c)--(e)  node[below right,midway]{$\kappa_t$};
\draw (c)--(f)  node[ above right,midway]{$\kappa_s\times \kappa_q$};
}\]
Thus, we have that 
\begin{align*}
  \Res_{\mu',\mu''}(\cM'\star \cM'')&= (\kappa_s\times \kappa_q)_!
  \kappa_t^!(\pi_t)_*(\pi_s\times \pi_q)^*(\cM'\star \cM'')\\ 
& = (\kappa_s\times \kappa_q)_!  (\tilde \pi_t)_*\tilde \kappa_t^!
  (\pi_s\times \pi_q)^*(\cM'\star \cM'')
\end{align*}
The variety $B$ can be stratified into subsets $B_\tau$ according to the dimension $\tau$ of the
intersection between the subrepresentations of dimension $\nu'$ and
$\mu'$.  Intersection with the other subrepresentation induces subs of
dimension $\tau$ in 
$\pi_s\tilde \kappa_t$ and $\kappa_s\tilde \pi_t$, and taking its
image induces a subs of dimension $\mu'-\tau$ in  $\pi_s\tilde
\kappa_t$ and dimension $\nu'-\tau$ in $\kappa_s\tilde \pi_t$.  Let
$\tau'=\nu''+\mu''-\nu+\tau$.  The
map from $B_\tau$ to the fiber product of $E_{\tau;\nu'-\tau}\times
E_{\mu'-\tau; \tau'}$ with $E_{\tau;\mu'-\tau}\times
E_{\nu'-\tau; \tau'}$ over $E_{\tau}\times E_{\mu'-\tau}\times
E_{\nu'-\tau}\times E_{\tau'}$ is an affine bundle of dimension
$\langle \tau+\tau',\mu'+\nu'-2\tau\rangle$.  Thus, 
\[\Delta_{\mu',\mu''}([\cM']\star[\cM''])=
\sum_{\tau}\Delta_{\tau;\nu'-\tau}([\cM'])\star \Delta_{\mu'-\tau;\tau'}[\cM''].\qedhere \]
\end{proof}

\subsection{Hall algebras}
\label{sec:hall-algebras}

While the previous section interpreted the weighted KLR algebras in
terms of characteristic 0 geometry, we can also consider the geometry
of quivers over a field of characteristic $p$.  The varieties $E_\nu,
X_\Bi$ and the algebraic group $G_\nu$ are all defined as $\Z$-schemes
whose base change to $\C$ are the varieties considered in the previous
sections.  After base change to $\Fqb$ for $q$ a prime
power, we can use the same pushforwards to define $\ell$-adic sheaves
$Y_\Bi$, which we denote with the same symbol as the corresponding
sheaves over $\C$; in this section, we will always consider sheaves on
varieties over $\Fqb$, so there is no danger of
confusion. By the usual comparison theorems in \'etale geometry (for
example, \cite[6.1.9]{BBD}), the
Ext-algebra of the sum of these sheaves is $\walg^\vartheta_\nu$, just as
it is for sheaves over $\C$.

The sheaves $Y_\Bi$ have a unique mixed structure which is pure of
weight 0.  As always, the pushforward by a proper map of the constant
sheaf with it canonical weight 0 mixed structure is again pure of
weight 0. If we apply the shift in the derived category without
changing the action of Frobenius, we will change the weight, but we
can apply a Tate twist to return to weight 0.  We will always take
this mixed structure.  In this section, the functor $\gamma$ will land
in this category, not its characteristic 0 analogue.  
The proof of Propositions \ref{pro:monoidal} and
\ref{pro:sheaf-bialgebra} carry over without change.

The reader might thus justly wonder what is achieved by introducing
this more difficult formalism.  Our primary motivation is a better
understanding of the Grothendieck group $K$.  
Recall that for any finite field $\mathbb{F}_q$, there is a {\bf Hall
  algebra} $\mathcal{H}_{\Gamma;q}$ of representations of $\Gamma$,
the space of all $\K$-valued function on the set of isomorphism
classes of quiver representations over $\mathbb{F}_q$.
We refer to the notes of Schiffmann  \cite{SchiffHall} for basic facts and definitions of
Hall algebras, but our Hall algebra will have the opposite product and coproduct
from Schiffmann's for compatibility with our diagrammatic formulation.
In essence, this is because our conventions are adapted to writing
short exact sequences with arrows pointing left to right (as any
right-thinking person would).

Attached to any mixed
complex of sheaves $\cM$ over an extension $\K$ of $\Q_\ell$ on
$E_\nu$, we have a function $\func_\cM\colon E_\nu (\Fq)\to \K$
sending $e\in E$ to the supertrace of the Frobenius morphism acting on
the stalk at that point:
\[\func_\cM(e):=\sum_{i\in \Z}(-1)^i\Tr(\operatorname{Fr}\mid
H^i(\cM_e)).\]  If we let $\sheafK$ denote the Grothendieck group of
the category of pure weight 0 shifts of perverse sheaves over $\K$,
then $\func_\cM\colon \sheafK\to \mathcal{H}_{\Gamma;q}$.

\begin{proposition}\label{pro:func-sheaf}
  The map $\func_\cM\colon \sheafK\to \mathcal{H}_{\Gamma;q}$ is a map
  of bialgebras.
\end{proposition}
\begin{proof}
  This follows instantly from the Grothendieck trace formula.
\end{proof}
While the definition of these functions may sound awfully abstruse, for geometrically natural sheaves, these functions are quite
explicit.  Of greatest importance to us is that
\begin{proposition}
$\displaystyle \func_{Y_\Bi}(e)=q^{\nicefrac{\Bu(\Bi)}{2}}\cdot \#\{x\in
X_{\Bi}(\mathbb{F}_q)\mid p(x)=e\}$
\end{proposition}
\begin{proof}
  This follows immediately from the Grothendieck trace formula; the
  factor of $q^{\nicefrac{\Bu(\Bi)}{2}}$ comes from the necessary
  Tate twist.
\end{proof}
Combining these propositions, we obtain the relationship between the
Grothendieck group $K^\vartheta$ and the Hall algebra.

\begin{proposition}
  There is natural map of Hopf algebras (in the braided category of
  $\Z[I]$-graded abelian groups) from $h_q\colon K^\vartheta\to
  \mathcal{H}_{\Gamma;q}$.  The induced map $\prod_{q^n}
  h_{q^n}\colon K^\vartheta\to
  \prod_{n\geq 1}\mathcal{H}_{\Gamma;q^n}$ is injective.
\end{proposition}
\begin{proof}
Since all these properties descend automatically to any subfield, and
hold for all algebraically closed fields of characteristic 0 if they
hold for one, we may assume that $\K=\mathbb{\bar Q}_\ell$ for some
prime $\ell$ coprime to $p$.

The map $h_q$ is the composition of that induced by $\gamma$ and
$\func_*$.  This is a map of bialgebras by Propositions
\ref{pro:sheaf-bialgebra} and \ref{pro:func-sheaf}. If we have a
non-zero element of the kernel, it corresponds to a non-zero linear
combination of pure complexes, and thus a pair of pure complexes which
are not isomorphic, but give the same function for infinitely many
powers of the same prime $p$;  this is impossible by
\cite[Th\'eor\`eme 1.1.2]{Lau} 
\end{proof}

This theorem, in particular, connects together the categorification
theorem for $U_q(\mathfrak{n})$ given by Khovanov and Lauda
\cite[3.18]{KLI}, and the result of Ringel giving an isomorphism
between $U_q(\mathfrak{n})$ and the composition subalgebra of the Hall
algebra \cite{Ringel} by giving a canonical isomorphism between
$K^0_q(R_\nu)$ and the composition algebra in $\mathcal{H}_{\Gamma;q}$
without passing through quantum groups.  This picture could easily
worked out by an expert from the paper of Varagnolo and Vasserot
\cite{VV}, but we know of nowhere where it was written explicitly. 

This relation to the Hall algebra gives a concrete approach to
computing the Grothendieck group of weighted KLR algebras.  For
example, when $\Gamma$ is affine, we obtain an isomorphism between
$K^0_q(\walg^\vartheta)$ for $k>0$ with the subalgebra of the Hall
algebra with nilpotent support considered by Vasserot and Varagnolo,
amongst others \cite{MR1722955}. 

 \bibliography{../gen}
\bibliographystyle{amsalpha}

\end{document}